\documentclass[12pt]{article}
\title{Well-posedness and long-term behaviour for a troposphere wave propagation model
}
\author{Paul Holst\thanks{Universität Hamburg, MIN faculty, Department of Mathematics, Germany ({paul.holst@uni-hamburg.de}).}
\and Jens D.M. Rademacher\thanks{Universit\"at Hamburg, MIN faculty, Department of Mathematics, Germany ({jens.rademacher@uni-hamburg.de}).} 
}
\date{}
\usepackage[utf8]{inputenc}
\usepackage{amsmath}
\usepackage{amsthm}
\usepackage{amssymb}
\usepackage{mathtools}
\usepackage{stmaryrd}
\usepackage{bbm}

\usepackage{geometry}

\usepackage[pdfpagelabels=false,colorlinks=true,allcolors=black]{hyperref}  % \label, \ref and \cite are recommended
\usepackage{todonotes}

\numberwithin{equation}{section}

\theoremstyle{definition}
\newtheorem{rem}{Remark}[section]
\theoremstyle{plain}  
\newtheorem{thm}[rem]{Theorem}
\newtheorem{defn}[rem]{Definition}
\newtheorem{prop}[rem]{Proposition}
\newtheorem{lem}[rem]{Lemma}

\newcommand{\R}{\mathbb{R}}

\newcommand{\N}{\mathbb{N}}

\newcommand{\dx}{\, dx}
\newcommand{\vr}{\varphi}
\renewcommand{\phi}{\varphi}
\newcommand{\ldiv}{L_{2,\textnormal{div}}(\Omega)}

\newcommand{\dom}{\operatorname{dom}}

\newcommand{\pl}{\partial}
\newcommand{\hhp}{H_{\textnormal{hp}}^1(\Omega)}

\newcommand{\hhpu}{H_{\textnormal{div}}^1(\Omega)}
\newcommand{\hhpui}{H_{\textnormal{div}}^{-1}(\Omega)}

\newcommand{\chp}{C_{\textnormal{hp}}^{\infty}(\overline{\Omega})}

\newcommand{\Hhpu}{H_{\textnormal{div}}^2(\Omega)}
\newcommand{\chpu}{C_{\textnormal{div}}^{\infty}(\overline{\Omega})}

\newcommand{\lin}{\operatorname{lin}}
\newcommand{\ind}{\mathbf{1}}
\newcommand{\Dp}{\Delta_{\textnormal{div}}}

\newcommand{\sq}{\subseteq}

\definecolor{colorJRblue}{rgb}{0.,0.,1.}%67}
\definecolor{colorJRred}{rgb}{1.,0.,0.}%67}
\definecolor{colorJRgreen}{rgb}{0.,0.6,0.}%67}

\begin{document}

\maketitle

\begin{abstract}
In this paper, we investigate a model recently derived by A. Constantin and R.S. Johnson for nonlinear wave propagation in the troposphere, particularly the 'morning glory' cloud pattern. We consider the model with natural Dirichlet boundary conditions for the vertical velocity at the top of the troposphere, and thus introduce a new pressure term. This modified system has a structural relation to the 2D primitive equations, for which global well-posedness and the existence of a global attractor are already known. We transfer these results to the modified model, giving proofs that exploit specific features and use standard methods combined with anisotropic Sobolev inequalities. Additionally, we show that the attractor exists only for specific parameter ranges, while for other parameters, we find runaway solutions with unbounded growth over time.   

\vspace{8pt}


%35D35 strong solutios to PDE

\vspace{2pt}

\end{abstract}

%Journals: 

\section{Introduction}\label{sec-intro}

In the context of nonlinear wave propagation in the troposphere, in \cite{CoJo22} Constantin and Johnson derived an asymptotic model from the general equations governing the atmospheric flow. Following \cite{MaRo23}, we write the resulting system of partial differential equations from \cite[Equation~(6.18)]{CoJo22} in the simplified form 
\begin{align}
\pl_t u + u\pl_1u + v \pl_2u &= \mu \Delta u + \alpha u + \beta v + K, \label{main-sys}\\
\pl_1 u + \pl_2 v &= 0, \label{incom-cond}
\end{align}
which are solved for $u,v \colon [0,T] \times \Omega \to \R$ in a suitable rectangular domain $\Omega \sq \R^2$. 
Here $T, \mu > 0$, $\alpha, \beta \in \R$, and $K \colon [0,T] \times \Omega \to \R$ is a given function{, representing thermodynamic forcing.}

When keeping the boundary conditions flexible, system \eqref{main-sys}-\eqref{incom-cond} {possesses} a special oscillatory solution discovered in \cite[Section~6(d)]{CoJo22} and traveling wave solutions found in \cite{CoJo23}. 
Implementing a Dirichlet boundary condition for $u$ and $v$ at the lower boundary of $\Omega$ and for $u$ only at the top boundary of $\Omega$, the existence of weak solutions and short-time classical solutions with $\Omega = \R \times (0,1)$ was shown in \cite{MaRo23} and the global existence of strong solutions for small inital data was established in \cite{AlGr24}. 
We are not aware of further results regarding well-posedness of \eqref{main-sys}-\eqref{incom-cond}. 

Motivated by these studies and a possible relation to the standard atmospheric model given by the primitive equations, in this paper we reconsider the boundary conditions. Dirichlet boundary condition for $u$ and $v$ at the lower boundary are meaningful, as it represents the Earth's surface, where no-slip conditions apply. 
%
%From a physical perspective, following common atmospheric modelling, it makes sense to add a Dirichlet boundary condition for $v$ at the upper boundary of the domain. 
%Indeed, in the derivation of the model, $u$ and $v$ represent the horizontal and vertical velocity components multiplied by the air density, respectively. 
Considering that the air density is low at the height of the upper boundary of the troposphere, it is typically the case that $u$ and $v$ are also small at this height, and a standard assumption  is that these $0$ at the upper boundary. See, e.g., \cite[\S 11.3.2]{Pie02} and \cite{Kelly} also for a scrutiny of these boundary conditions. 
Imposing the Dirichlet boundary condition for $v$ at the upper boundary fixes the flow and to maintain this constraint, the system responds by introducing an additional pressure term. This is similar to what happens in the 2D primitive equations for the ocean, where the pressure adjusts the flow to satisfy the Dirichlet boundary boundary condition at the ocean surface for the vertical velocity $w$. 

When implementing these modifications for a fixed domain 
\[
\Omega:=(-\pi,\pi) \times (0,\pi),
\]
\eqref{main-sys}-\eqref{incom-cond} turn into the new system of partial differential equations \begin{align}
\pl_t u + u\pl_1u + v \pl_2u + \pl_1p &= \mu \Delta u + \alpha u + \beta v + K, \label{main-sys-2}\\
\pl_1 u + \pl_2 v &= 0, \label{incom-cond-2}
\end{align}
where we seek $u,v\colon [0,T] \times \Omega \to \R$ and {pressure} $p\colon [0,T] \times \Omega \to \R$ 
subject to the boundary conditions 
\begin{align}
&u,v \text{ are periodic in } x_1 \text{ of period } 2\pi, \label{bound-cond-1} \\ &u(\cdot,0) = u(\cdot,\pi) = 0 \text{ on } (-\pi,\pi), \quad  v(\cdot,0) = v(\cdot,\pi) = 0 \text{ on } (-\pi,\pi), \label{bound-cond-2}   
\end{align}
and initial condition 
\begin{equation}\label{init-cond}
u(0) = u_0,    
\end{equation}
with $u_0 \colon \Omega \to \R$ a given function.

In this paper we study the well-posedness and aspects of the long term dynamics of system \eqref{main-sys-2}--\eqref{init-cond}. 
In particular, we show global well-posedness, i.e., the existence of unique global strong solutions and the continuous dependence on the initial value, for the system \eqref{main-sys-2}-\eqref{init-cond}, which also corroborates 
the need of the additional pressure term. 
Concerning the long-term behaviour we first note that both \eqref{main-sys}-\eqref{incom-cond} and \eqref{main-sys-2}-\eqref{init-cond} possess the explicit plane wave solution 
\[
u(t,x)=e^{(\alpha-\mu)t}\sin(x_2).
\]
For $\alpha>\mu$ this grows exponentially and unboundedly, similar to \cite{Calz}, but might also be unphysical and related to reflections from the boundary \cite{Kelly}. The presence of these runaway solutions suggests that care has be taken concerning the values for the parameters $\alpha$, $\beta$ and $\mu$. Indeed, we show this behaviour is not occurring in \eqref{main-sys-2}-\eqref{init-cond} for $\alpha$ below an explicit bound in terms of $\beta,\mu$, for which we prove the existence of a compact global attractor. 

%JR: Values of alpha, beta hard to understand from \cite{CoJo22}. At least alpha >0 can be seen from the values of C, S, sigma, and their alpha before their (6.2), which go into their (5.28). It is however unclear what other values are such as rho_0, which in their case seems to be a function while we fixed it. Also unclear is the viscosty parameter mu. In their (5.28) the Laplace operator Delta_0 is actually anistropic due to Eddy viscosities. But such anisotropy in our (1.3) can be scaled out: rescaling \pa_2 to get an isotropic Laplacian introduces a factor before the v \pa_2 u term that can be removed by scaling v also in (1.4), and can be absorbed in beta.

The proofs of global well-posedness and the global attractor are essentially a direct application of the methods for similar results for the 2D primitive equations. The key observation for this relation is that at $\alpha = \beta = 0$  \eqref{main-sys-2}-\eqref{init-cond} is a subsystem of the 2D primitive equations. In our proofs of both results we show how to deal with non-zero $\alpha$ and $\beta$. We build upon the approach in \cite{PeTeZi09} to establish well-posedness via a-priori estimates, adapting the method by leveraging the improved behavior of the nonlinearity under Dirichlet boundary conditions for $u$ and $v$ at the top and bottom of the domain, as opposed to the Neumann boundary conditions used in \cite{PeTeZi09}. This modification allows us to employ anisotropic a priori Sobolev estimates that are not applicable in their work, thereby simplifying our proof compared to that in \cite{PeTeZi09}. These estimates may also be of independent interest. 

\medskip
Let us briefly return to the modelling background. As noted above, the model \eqref{main-sys}-\eqref{incom-cond} has been derived in order to better understand nonlinear wave propagation in the troposphere. This system especially serves as a  mathematical model for the `morning glory' cloud pattern, an awe-inspiring atmospheric phenomenon predominantly observed in coastal regions, particularly in Australia. The morning glory cloud pattern involves the mesmerizing movement of long, slender cloud formations that extend horizontally for hundreds of kilometers, perpendicular to the cloud line. For more details about this phenomenon we refer to 
\cite{BiRe13,Chr92,CrRoSm82}. 
The motion captured by \eqref{main-sys}-\eqref{incom-cond} is essentially two-dimensional, occurring in the horizontal $x_1$- and vertical $x_2$-directions. In this context, the variables $u$ and $v$ can be interpreted as the corresponding horizontal and vertical components of velocity, respectively. 
The term $K$ represents a thermodynamic forcing term that encompasses the heat sources responsible for driving the observed motion. 
We note that \eqref{incom-cond} does not result from an incompressibility assumption. 
For a more comprehensive understanding of the physical background and the derivation of this model we refer to \cite{CoJo22}. 

\medskip 
The remainder of this paper is organised as follows. In \S\ref{sec-main} we formulate the two main results in mathematical detail: global well-posedness of \eqref{main-sys-2}-\eqref{init-cond} and existence of a global attractor for suitable choices of $\alpha$, $\beta$ and $\mu$. Section~\ref{sec-prelim} provides essential preliminaries for proving these results. Finally, \S\ref{sec-strsol} and \S\ref{sec-attractor} contain the proofs.

\section{Main results}\label{sec-main}

In order to formulate our main results, Theorems~\ref{main-thm} and~\ref{main-thm-attrac} below, we rewrite  
\eqref{main-sys-2}-\eqref{init-cond} as an evolution equation only in terms of $u$ 
and provide definitions for strong and weak solutions. We first introduce appropriate function spaces 
\begin{align*}
\chp &:= \{ u \in C^{\infty}(\overline{\Omega}); \ \pl^{\gamma}u(-\pi,x_2) = \pl^{\gamma}u(\pi,x_2), \forall x_2 \in [0,\pi], \ \gamma \in \N_0^2 \}, \\
\chpu &:= \bigg\{ u \in \chp; \ u(\cdot,0) = u(\cdot,\pi) = 0, \ \int_0^{\pi} \pl_1 u(\cdot,x_2) \,dx_2 = 0 \bigg\}, \\
H_{\textnormal{hp}}^k(\Omega) &:= \overline{\chp}^{H^k(\Omega)} \qquad (k \in \N), \\ 
H_{\textnormal p}^1(-\pi,\pi)&:=\overline{\{u \in C^{\infty}[-\pi,\pi]; \ u^{(n)}(-\pi) = u^{(n)}(\pi) \ (n \in \N)\}}^{H^1(-\pi,\pi)}\\
\ldiv &:= H_{\textnormal{div}}^0(\Omega):= \bigg\{u \in L_2(\Omega);  \int_{\Omega} u \pl_1 \psi \dx = 0 , \forall\psi \in H_{\textnormal p}^1(-\pi,\pi) \bigg\}, \\
H_{\textnormal{div}}^k(\Omega) &:= \overline{\chpu}^{H^k(\Omega)} \qquad (k \in \N).
\end{align*}
Here $\chpu$, $\ldiv$ and $H_{\textnormal{div},0}^k(\Omega)$ denote the spaces of all $\chp$-, $L_2(\Omega)$- and $H_{\textnormal{hp}}^k(\Omega)$-functions, respectively, {with the following properties: The functions and} all their derivatives are periodic with period $2\pi$ in the $x_1$-variable, satisfy a Dirichlet boundary condition on $[-\pi,\pi] \times \{0,\pi\}$ and the integral condition $\int_0^{\pi} \pl_1 u(\cdot,x_2) \,dx_2 = 0$. The latter condition is equivalent to a Dirichlet boundary condition for the vertical velocity $v$ 
(see \eqref{def-v} below). We note that $u \in \ldiv$ satisfies this condition in a \textit{weak sense} only, because $u \in \hhp$ implies $\int_0^{\pi} \pl_1 u(\cdot,x_2) \,dx_2 = 0$ if and only if 
\[
\int_{\Omega}u \pl_1 \psi\dx = - \int_{-\pi}^{\pi} \int_0^{\pi} \pl_1 u(x_1,x_2) \,dx_2 \, \psi(x_1) \,dx_1 = 0 \qquad \forall\psi \in H_{\textnormal p}^1(-\pi,\pi).
\]
For every $k \in \N$, the spaces $H_{\textnormal{hp}}^k(\Omega)$ and $H_{\textnormal{div}}^k(\Omega)$ are Hilbert spaces with the scalar product induced by $H^k(\Omega)$. In what comes, we will also use the fact that the mapping $\|\cdot\|_{1,\textnormal{div}} \colon \hhpu \to \R,$ 
\begin{equation}\label{def-norm-div}
\|u\|_{1,\textnormal{div}}:= \|\nabla u\|_2 \qquad (u \in \hhpu)
\end{equation}
defines an equivalent norm on $\hhpu$; see \eqref{poin-1} below. Moreover we denote by $\hhpui$ the dual space of $\hhpu$. In the following, we will also consider $u \in L_2(\Omega)$ as an element of $\hhpui$, by virtue of the embedding $L_2(\Omega) \ni u \mapsto (v \mapsto  \langle u,v\rangle_2) \in \hhpui$, where here and in the following, we write $\langle f,g \rangle_{L_2(\Omega)}:= \langle f,g \rangle_2 := \int_{\Omega} fg \dx$ for the scalar product in $L_2(\Omega)$ of two functions $f,g \in L_2(\Omega)$. Note that $\|u\|_{\hhpui} \leqslant \|u\|_2$ for all $u \in L_2(\Omega)$. 

Furthermore, for $k \in \N_0$, we observe that 
\begin{align}
H_{\textnormal{div}}^k(\Omega) &= \bigg\{ u \in H_{\textnormal{hp}}^k(\Omega); \ u(\cdot,0) = u(\cdot,\pi) = 0, \ \int_0^{\pi} \pl_1 u(\cdot,x_2) \dx_2 = 0 \bigg\} \nonumber\\ 
&= \overline{\lin\{e_{k,\ell} ; \ \ell \in \mathcal{Z} 
\} }^{H^k(\Omega)}, \label{H-div-identity}
\end{align}
where {$\mathcal{Z}:= (\mathbb{Z} \times 2\N) \cup (\{0\} \times \N)$ and}
\begin{align}
e_{k,\ell}(x_1,x_2)&:=a_{k,\ell}\cos(\ell_1x_1)\sin(\ell_2x_2) \qquad (\ell \in \N_0 \times \N), \label{ortho-1}\\ 
e_{k,\ell}(x_1,x_2)&:=b_{k,\ell}\sin(\ell_1x_1)\sin(\ell_2x_2) \qquad (\ell  \in (\mathbb{Z} \setminus \{\N_0\}) \times \N), \label{ortho-2}
\end{align}
with $a_{k,\ell},b_{k,\ell} \in \R$ chosen such that $\|e_{k,\ell}\|_{H^k(\Omega)} = 1$ for all $\ell \in \mathcal{Z}$
; in particular,
\begin{equation}\label{l2-div-identity}
\ldiv = \overline{\lin\{e_{0,\ell} ; \ \ell \in \mathcal{Z}
%(\mathbb{Z} \times 2\N) \cup (\{0\} \times \N)
\} }^{L_2(\Omega)}.    
\end{equation}
These identities hold since $(e_{k,\ell})_{\ell \in \mathcal{Z}}$ is an orthonormal basis of $H_{\textnormal{div}}^k(\Omega)$, 
which follows from the standard fact that 
it is for $\{u \in H_{\textnormal{hp}}^k(\Omega) \ u(\cdot,0) = u(\cdot,\pi) = 0\}$ and the equivalence
\[
\ell \in \mathbb{Z} \times \N, \ \int_0^{\pi} \pl_1 e_{k,\ell}(\cdot,x_2) \dx_2 = 0 \quad \Longleftrightarrow \quad \ell \in (\mathbb{Z} \times 2\N\}) \cup (\{0\} \times \N).
\]
To simplify notation, let $e_{\ell}:= e_{0,\ell}$ for all $\ell \in \mathbb{Z} \times \N$.

In the following we aim to express the system  \eqref{main-sys-2}-\eqref{init-cond} as an evolution equation only in terms of $u${, which will involve a Leray-projection as usual for incompressibility.  First, to} represent $v$ in terms of $u$, we {integrate} \eqref{incom-cond} 
with respect to $x_2$ and use the boundary conditions \eqref{bound-cond-2}  and obtain 
\[
v{(x_1,x_2)} = - \int_0^{x_2} \pl_1u{(x_1,x_2')} \dx_2'.
\]
Hence, we define for $u \in L_2(\Omega)$ with $\pl_1 u \in L_2(\Omega)$, 
the function
\begin{equation}\label{def-v}
v_u(x_1,x_2):= - \int_0^{x_2} \pl_1 u(x_1,x_2') \,dx_2' \qquad (\text{a.e. } (x_1,x_2) \in \Omega), 
\end{equation}
so that $v_u \in L_2(\Omega)$ and $\pl_2 v_u =  -\pl_1 u$. We note that  $v_u(x_1,\cdot) \in H_0^1(0,\pi)$ for a.e. $x_1 \in (-\pi,\pi)$ if $u \in \hhpu$.

We seek solutions $u$ of \eqref{main-sys-2} such that $\int_0^{\pi} \pl_1u(\cdot,x_2) \,dx_2 = 0$, 
and, using $\pl_1 p$, we can choose $K$ to satisfy this. However, the constraint is in general violated by the nonlinear terms $u\pl_1u$, $v_u\pl_2u$ as well as the linear term $\beta v_u$ in \eqref{main-sys-2}. 
Hence, we use the orthogonal projection $\mathbb{P}\colon L_2(\Omega) \to \ldiv$ onto $\ldiv$ 
that will in particular remove the pressure and is in fact the first component of the usual Leray-projection. For later purposes we provide some details next. 
Due to \eqref{l2-div-identity}, the orthogonal complement 
of $\ldiv$ is given by
\[
\ldiv^{\perp} = \overline{\lin\{e_{\ell} ; \ \ell \in (\mathbb{Z} \setminus \{0\}) \times (2\N-1))\} }^{L_2(\Omega)}.
\]
On the other hand one readily checks that 
\[
\overline{\lin\{e_{\ell} ; \ \ell \in (\mathbb{Z} \setminus \{0\}) \times (2\N-1))\} }^{L_2(\Omega)} \subseteq \{\pl_1p; \ p \in {H_{\textnormal{hp}}^{1,0}(\Omega) } \cap \ldiv^{\perp}\} \sq \ldiv^{\perp},
\]
where here {$H^{1,0}(\Omega):=\{u \in L_2(\Omega); \ \pl_1 u \in L_2(\Omega)\}$}. Hence, 
 $\ldiv^{\perp} = \{\pl_1p; \ p \in {H_{\textnormal{hp}}^{1,0}(\Omega)} \cap \ldiv^{\perp}\}$ 
 and thus, for every $u \in L_2(\Omega)$, one has that $(I-\mathbb{P})u = \pl_1 p$ for some $p \in {H_{\textnormal{hp}}^{1,0}(\Omega) } \cap \ldiv^{\perp}$.

\medskip
For such $p$, the representation \eqref{def-v} {the assumption $K \in L_2(0,T;\ldiv)$} and the orthogonal projection $\mathbb{P}$, we can rewrite \eqref{main-sys-2}-\eqref{init-cond} as the evolution equation
\begin{equation}\label{evol-1}
u'(t) + \mathbb{P}(u(t)\partial_1 u(t) + v_{u(t)}\partial_2u(t)) = \mu\Delta u(t) + \alpha u(t) + \beta\mathbb{P} v_{u(t)} + K(t), \quad u(0) = u_0,
\end{equation}
where $u'$ represents the derivative of $u$ with respect to $t$. 

Concerning strong solutions, {we can now mathematically define them as follows.} 
\begin{defn}
Let $u_0 \in \ldiv$ and $K \in L_2(0,T;\ldiv)$. We call a function $u \colon (0,T) \to \ldiv$ a \textbf{strong solution of \eqref{evol-1}} (with respect to $u_0$ and $K$) if $u \in L_2(0,T;H_{\textnormal{div}}^2(\Omega)) \cap H^1(0,T;\ldiv)$, $u(0) = u_0$ and
\eqref{evol-1} holds in $\ldiv$ for a.e. $t \in (0,T)$. 
\end{defn}
Note that such $u$ automatically lies in $C([0,T];\hhpu)$, since $L_2(0,T;H_{\textnormal{div}}^2(\Omega)) \cap H^1(0,T;\ldiv)$ is continuously embedded in $C([0,T];\hhpu)$, as shown in \cite[Thm~19.7]{ArVoVoi22} (applied with $A = \Dp $, which is defined at the beginning of Section~\ref{sec-prelim}).

\medskip
In order to provide the definition of a weak solution we first derive the weak formulation of \eqref{evol-1} assuming that $u$ is a classical solution of \eqref{evol-1}: 
multiply \eqref{evol-1} by 
$\varphi \psi$, where $\varphi \in C^\infty[0,T)$ with compact support in $[0,T)$ and $\psi \in \chpu$, integrate the resulting equation over $(0,T) \times \Omega$, and 
integrate by parts. We thus obtain the weak formulation
\begin{align*}
& - \int_0^T \langle u(t), \vr'(t)\psi \rangle_2 \,dt + \mu\int_0^T \int_{\Omega} \nabla u(t) \cdot \nabla \psi + (u(t) \pl_1u(t) + v_{u(t)}\pl_2u(t)) \psi  \dx \,\vr(t) \,dt \\ 
&= \int_0^T \alpha \langle u(t),\vr(t)\psi \rangle_2 + \beta \langle v_{u(t)},\vr(t) \psi \rangle_2 \,dt + \int_0^T \langle K(t),\vr(t)\psi \rangle_2 \,dt + \langle u_0,\vr(0)\psi \rangle_2.
\end{align*} 
This motivates our definition of a weak solution for \eqref{evol-1}.
\begin{defn}
Let $u_0 \in \ldiv$ and $K \in L_2(0,T;H_{\textnormal{div}}^{-1}(\Omega))$. We call a function $u\colon (0,T) \to \ldiv$ a \textbf{weak solution of \eqref{evol-1}} (with respect to $u_0$ and $K$) if $u \in L_2(0,T;\hhpu)$, $u' \in L_1(0,T;H_{\textnormal{div}}^{-1}(\Omega))$ and
\begin{align}
&- \langle u_0,\vr(0)\psi \rangle_2 - \int_0^T \langle u(t), \vr'(t)\psi \rangle_{H_{\textnormal{div}}^{-1}(\Omega),H_{\textnormal{div}}^1(\Omega)} \,dt \nonumber \\
& + \mu\int_0^T \int_{\Omega} \nabla u(t) \cdot \nabla \psi + (u(t) \pl_1u(t) + v_{u(t)}\pl_2u(t))\psi \dx \, \vr(t) \,dt \nonumber \\ 
&= \int_0^T \alpha \langle u(t),\vr(t)\psi \rangle_2 + \beta \langle v_{u(t)},\vr(t) \psi \rangle_2 \,dt + \int_0^T \langle K(t),\vr(t)\psi \rangle_{H_{\textnormal{div}}^{-1}(\Omega),\hhpu} \,dt \nonumber
\end{align}
holds for all $\varphi \in C^\infty[0,T)$ with compact support in $[0,T)$ and all $\psi \in \hhpu$ (here $\langle \cdot,\cdot \rangle_{H_{\textnormal{div}}^{-1}(\Omega),\hhpu}$ denotes the dual pairing of $\hhpu$ and $H_{\textnormal{div}}^{-1}(\Omega)$)
\end{defn}

\medskip
Now we are ready to formulate the result about global well-posedness for the system \eqref{main-sys-2}-\eqref{init-cond}, our first main result.
{Here and in the following we freely use $v$ to denote a function, not the velocity component, which is given as $v_u$ from \eqref{def-v}.}

\begin{thm}\label{main-thm}
Let $T, r > 0$, and let $u_0 \in L_2(\Omega)$, $K \in L_2(0,T;H_{\textnormal{div}}^{-1}(\Omega))$ such that $\|u_0\|_2 \leqslant r$ and $\|K\|_ {L_2(0,T;H_{\textnormal{div}}^{-1}(\Omega))} \leqslant r$. Then there exists a weak solution of \eqref{evol-1} with respect to $u_0$ and $K$. Moreover, if $u_0 \in \hhpu$ and $K \in L_2(0,T;\ldiv)$ with $\|\nabla u_0\|_2 \leqslant r$ and  $\|K\|_ {L_2(0,T;L_2(\Omega))} \leqslant r$, then the weak solution is unique. 

Furthermore, for $v_0 \in \hhpu$ with $\|\nabla v_0\|_2 \leqslant r$ the corresponding strong solution $v$ of \eqref{evol-1} with respect to $v_0$ and $K$ satisfies
\begin{align}
\sup_{t \in [0,T]}\|\nabla (u(t) - v(t))\|_2^2 
\leqslant C \|\nabla(u_0-v_0)\|_2^2  \label{contin-depend}
\end{align}
for some $C \geqslant 0$ only depending on $\alpha$, $\beta$, $\mu$, $T$ and $r$. In particular, the function $\hhpu \ni {u}_0 \mapsto {u} \in C([0,T];\hhpu)$ is Lipschitz continuous on bounded subsets of $\hhpu$.
\end{thm}

We note that {for $\alpha=\beta=0$} the equations \eqref{main-sys-2}-\eqref{init-cond} are a subsystem of the 2D primitive equations, e.g. \cite[(3.129)-(3.135)]{PeTeZi09}. 
Indeed, in the latter with zero rotation $f=0$ and forcing $F_v=F_S=F_T=0$, the equations preserve vanishing horizontal velocity component $v$ in $x_2$-direction, temperature $T$, salinity $S$. This gives a time-invariant subsystem that equals \eqref{main-sys-2}-\eqref{init-cond}, when replacing \cite[(3.137)]{PeTeZi09} by  horizontally periodic boundary conditions. 
This enables us to prove Theorem~\ref{main-thm} in a similar way as the global well-posedness of the 2D primitive equations (see \cite[Proof~of~Thm~3.3]{PeTeZi09}). 
In the course of proving Theorem~\ref{main-thm}, we will see that the terms $\alpha u$ and $\beta v$ can be readily handled 
and we obtain the crucial a-priori estimates needed for proving the existence of a global attractor in \S\ref{sec-attractor}. 

We briefly outline the proof of Theorem~\ref{main-thm}.
We use the usual Galerkin approximation procedure together with a suitable Aubin-Lions compactness theorem. In a first step, we establish an energy estimate and an a-priori estimate for $u'$ in $\hhpui$, which are shown in the standard manner and lead to the existence of weak solutions. To obtain existence of strong solutions, we then show an a-priori estimate for $\pl_2u$ using the fact that the nonlinearity multiplied by $\pl_2u$ and integrated over $\Omega$ vanishes (see Lemma~\ref{nonlin-eq-ineq}(c) below). Here, there is a slight difference from the proof for the 2D primitive equations, where 
the nonlinearity behaves `well', but does not vanish as neatly as for System \eqref{main-sys-2}-\eqref{init-cond}. This estimate together with an important estimate for the nonlinearity obtained by anisotropic Sobolev inequalities (see Lemma~\ref{mixed-space-ineq}, and also Lemma~\ref{nonlin-eq-ineq}(b)) are then used to obtain an a-priori estimate for $\partial_1 u$. Finally, with these estimates at hand, we prove a-priori estimates for $\Delta u$ and $u'$ in $L_2(\Omega)$, again following the standard approach.

\medskip
When examining the long-term behavior of the strong solutions to \eqref{main-sys-2}-\eqref{init-cond} obtained in Theorem~\eqref{main-thm},
{recall the observation mentioned in the introduction for $K = 0$: A straightforward computation shows that  } \eqref{evol-1} admits the one-dimensional space of solutions spanned by 
\begin{equation}\label{explicit-sol}
u(t,x) := e^{(\alpha-\mu)t}e_{(0,1)}(x) = a_{0,(0,1)} e^{(\alpha- \mu)t}\sin(x_2) \qquad (t \geqslant 0, \ x \in \Omega), 
\end{equation}
where $e_{(0,1)}${, defined in \eqref{ortho-1}, turns out to be} the first eigenfunction of $\Dp$. 
The $L_2(\Omega)$-norm of $u$ decays exponentially over time if $\mu > \alpha$,  remains constant if $\mu = \alpha$ (in this case $u$ is a \textit{steady state}) and grows exponentially and unboundedly over time with growth rate $\alpha-\mu$ if $\mu < \alpha${; we refer to the latter as `grow-up'}. In the latter case, the solutions of \eqref{main-sys-2}-\eqref{init-cond} accumulate infinite energy over time, which is clearly unphysical. 
In fact, all $e_{(0,\ell_2)}$, $\ell_2\in\N$, i.e., all $x_1$-independent eigenfunctions of $\Dp$ generate such linear invariant spaces that grow exponentially for sufficiently large $\alpha>\mu$. 
This suggests that in practice, one should be cautious with the choice of $\alpha$, $\beta$, and $\mu$. 
However, for $\pi|\beta| < \mu$ if $\alpha \leqslant 0$ and $\alpha + \pi |\beta| < \mu$ if $\alpha > 0$, we prove the existence of a global attractor in Theorem~\ref{main-thm-attrac} below, so in this case there is no problem of that kind. While for $\beta=K=0$ there is a sharp dichotomy between a compact attractor and grow-up, it remains unclear whether there are global bounds for $\alpha\in(\mu-\pi|\beta|,\mu)$. 

To formulate this result, we first introduce the semigroup associated with \eqref{evol-1}, which can be defined thanks to Theorem~\ref{main-thm}. Assuming $K \in L_{2,\textnormal{loc}}(0,\infty;\ldiv)$, for any $u_0 \in \hhpu$, it follows from Theorem~\ref{main-thm} that there exists a unique 
\begin{equation}\label{e:semigr}
S(\cdot)u_0:= u \in H_{\textnormal{loc}}^1(0,\infty;\ldiv) \cap C([0,\infty);\hhpu) \cap L_{2.\textnormal{loc}}(0,\infty;\Hhpu)
\end{equation}
satisfying \eqref{evol-1} with respect to $u_0$ and $K$. Let $M(A)$ denote the set of all functions from a set $A$ into itself. {We have that} $S\colon [0,\infty) \to M(\hhpu)$ defines a semigroup, i.e., $S(0)u_0 = u_0$ for all $u_0 \in \hhpu$ and $S(t+s)= S(t)S(s)$ for all $t,s \geqslant 0$. 
Moreover, due to \eqref{contin-depend}, the function $S(t)\colon \hhpu \to \hhpu$ is continuous. 

Our second main result, concerning the global attractor is now as follows.

\begin{thm}\label{main-thm-attrac}
Let $K \in H^1(0,\infty;\hhpui) \cap C_b([0,\infty);\ldiv)$, and assume that $\pi|\beta| < \mu$ if $\alpha \leqslant 0$ and $\alpha + \pi |\beta| < \mu$ if $\alpha > 0$. Then $S$ possesses a compact global attractor in $\hhpu$.
\end{thm}

As mentioned, for $\alpha = \beta = 0$,  \eqref{main-sys-2}-\eqref{init-cond} is a subsystem of the 2D primitive equations. Since these are in turn a subsystem of the 3D primitive equations,  Theorem~\ref{main-thm-attrac} follows for $\alpha=\beta=0$ from known global attractor results for the 3D primitive equations, cf.\ \cite[Thm.~4.3]{YouLi18}. 
For general $\alpha$ and $\beta$ -- as for the global existence -- we rely on standard methods to prove Theorem~\ref{main-thm-attrac}, highlighting how to exactly handle the additional linear terms $\alpha u$ and $\beta v$: we first show the boundedness in time $H^2$ of solutions (Proposition~\ref{main-thm-reg}), then use the uniform Gronwall lemma (Lemma~\ref{uniform-gronwall}) to show that $S$ possesses an absorbing ball in $\Hhpu$ norm (Proposition~\ref{main-thm-reg})
and finally infer the existence of the global attractor (Theorem~\ref{thm-attractor}).

\section{Preliminaries}\label{sec-prelim}
In this section we collect some preliminary results.

Throughout this paper we denote by $\Dp$ the Laplacian on $\ldiv$, i.e., $\Dp$ is an unbounded self-adjoint operator on $\ldiv$ defined by
\begin{align*}
\dom(\Dp)&:= \{u \in \hhpu; \ \Delta u \in L_2(\Omega)\}
, \\
\Dp u &:= \Delta u. 
\end{align*}
We observe that $(e_{\ell})_{\ell \in \mathcal{Z}}$ is an orthonormalbasis of $\ldiv$ constisting of eigenvectors of $-\Dp$ with corresponding eigenvalue 
given by $\lambda_{\ell}:= |\ell|^2$. 
Furthermore, $e_{\ell} \in \chpu$ and $\lambda_{\ell} \geqslant 1 = \lambda_{(0,1)}$ for all $\ell \in \mathcal{Z}$.
Using these facts, yields 
the Poincar\'e inequality 
\begin{equation}\label{poin-1}
\|u\|_2 \leqslant \|\nabla u\|_2 \qquad (u \in \hhpu)
\end{equation}
and the inequality
\begin{equation}\label{poin-2}
\|\nabla u\|_2 \leqslant \|\Delta u\|_2 \qquad (u \in \Hhpu);
\end{equation}
in particular, the mapping $\Hhpu \ni u \mapsto \|\Delta u\|_2 \in [0,\infty)$ defines an equivalent norm on $\Hhpu$. Thus, it is also not too difficult to show that $\dom(\Delta_{\textnormal{div}}) = \Hhpu$.

Next we want to prove 
estimates for the nonlinearity 
\[
B(u,v):= u \pl_1 v + v_u\pl_2v  \qquad (u,v \in H^1(\Omega)) 
\]
appearing in \eqref{evol-1}, that will be needed in Section~\ref{sec-strsol} to derive appropriate a-priori estimates for the proof of Theorem~\ref{main-thm}. %; for 
{Recall} the definition of $v_u$ { from \eqref{def-v}, and}
note that $v_u \in H^1(\Omega) \sq L_4(\Omega)$
and thus $B(u,v) \in L_2(\Omega)$ for all $u,v \in H^2(\Omega)$.

\begin{lem}\label{nonlin-eq-ineq}
(a) Let $u,v,w \in \hhpu$. Then
\[
\int_{\Omega} |B(u,v)w| \dx \leqslant M_B \|\nabla u\|_2\|\nabla v\|_2\|\nabla w\|_2
\]
for some $M_B \geqslant 0$ independent of $u$, $v$ and $w$.

\noindent (b) Let $u,v \in \Hhpu$ and $w \in L_2(\Omega)$. Then
\begin{align*}
\int_{\Omega} |B(u,v)w| \dx &\leqslant M_B'(\|u\|_2^{1/2}\|\nabla u\|_2^{1/2} \|\pl_1v\|_2^{1/2}\|\pl_1\nabla v\|_2^{1/2} \\ &\mkern20mu + \|\pl_1 u\|_2\|\pl_2 v\|_2^{1/2}(\|\pl_2v\|_2^{1/2} + \|\pl_1 \nabla v\|_2^{1/2})) \|w\|_2 \\
&\leqslant M_B'' \|\nabla u\|_2\|\nabla v\|_2^{1/2}\|\Delta v\|_2^{1/2} \|w\|_2 \leqslant M_B'' \|\nabla u\|_2\|\Delta v\|_2 \|w\|_2
\end{align*}
for some $M_B', M_B''\geqslant 0$ independent of $u$, $v$ and $w$.

\noindent (c) Let $u,v \in \hhpu$. Then $\int_{\Omega} B(u,v)v \dx = 0$. Moreover, if $u \in \Hhpu$, then 
$\langle B(u,u),\pl_2^2u \rangle_2 = 0$.
\end{lem}

\begin{rem} 
The estimate given in Lemma~\ref{nonlin-eq-ineq}(a) {is sufficient for our purposes, but we remark that it} can be improved, as we shall see in the proof of Lemma~\ref{nonlin-eq-ineq}(a) below; e.g., one even has that
\begin{align*}
\int_{\Omega} |B(u,v)w| \dx &\leqslant c(\|u\|_2^{1/2}\|\nabla u\|_2^{1/2} \|\pl_1v\|_2 \|w\|_2^{1/2} \|\nabla w\|_2^{1/2} \\ 
&\mkern20mu + \|\pl_1u\|_2\|\pl_2v\|_2 \|w\|_2^{1/2}(\|w\|_2^{1/2} + \|\pl_1 w\|_2^{1/2})) \qquad (u,v,w \in \hhpu)
\end{align*}
for some $c \geqslant 0$. 
\end{rem}

In preparation of the proof of Lemma~\ref{nonlin-eq-ineq}, we will now derive three important properties of the term $v_u$. The first property is the following rule of integration by parts:
\begin{equation}\label{int-parts-2}
\int_{\Omega} v_u\pl_2v\dx = \int_{\Omega} (\pl_1u)v \dx \qquad (u \in \hhpu, \ v \in \hhp);
\end{equation}
Indeed, if $u \in \chpu$ and $v \in \chp$, the equality \eqref{int-parts-2} directly follows from the fact that $v_u(x_1,\cdot) \in H_0^1(0,\pi)$ for a.e. $x_1 \in (-\pi,\pi)$, $\pl_2 v_u = - \pl_1 u$ and integration by parts. The general case is then immediately obtained by approximation. 

Furthermore, using H\"older's inequality, for all $u \in \hhpu$ we estimate
\begin{align}
\|v_u\|_2 &\leqslant \biggl( \int_{-\pi}^\pi \int_0^\pi \biggl( \int_0^\pi|\pl_1u(x_1,x_2')| \,dx_2' \biggr)^2 \,dx_2 \,dx_1\biggr)^{1/2} \nonumber \\
&\leqslant \sqrt{\pi}\biggl( \int_{-\pi}^\pi \int_0^\pi  |\pl_1u(x_1,x_2')|^2 \,dx_2' \int_0^\pi \ind_{(-\pi,\pi)}^2 \,dx_2'  \,dx_1\biggr)^{1/2} = \pi \|\pl_1u\|_2\label{v-l2-ineq}
\end{align}
and
\begin{equation}\label{v-linfty-ineq}
\|v_u(x_1,\cdot)\|_{L_{\infty}(0,\pi)} \leqslant \int_0^\pi |\pl_1u(x_1,x_2)| \,dx_2 \leqslant \sqrt{\pi}\|\pl_1 u(x_1,\cdot)\|_{L_2(-\pi,\pi)}
\end{equation}
for a.e. $x_1 \in (-\pi,\pi)$; these are the second and third property of $v_u$, respectively. 

For the proof of Lemma~\ref{nonlin-eq-ineq}, we also make use of the Ladyzhenskaya inequality
\begin{equation}\label{Lady-ineq}
\|u\|_4 \leqslant c_{\textnormal L}\|u\|_2^{1/2}\|\nabla u\|_2^{1/2} \qquad (u \in \hhpu),    
\end{equation}
with some $c_{\textnormal L} > 0$, which follows from \eqref{poin-1} and \cite[Theorem~5.8]{Ada75}. Moreover, the equality 
\begin{equation}\label{eq-Delta}
\|\nabla \pl_1u\|_2^2 + \|\nabla \pl_2 u\|_2^2 = \|\Delta u\|_2^2 \qquad (u \in \Hhpu)
\end{equation}
will prove to be useful; this directly 
follows from integration by parts for $u \in \lin\{e_{\ell}; \ \ell \in \mathcal{Z}\}$,
and by approximation for $u \in \Hhpu$.
The next lemma states an anisotropic Sobolev inequality, which has been proven in a similar form in \cite[Lemma~2.3]{CaLiTi17}. Our proof follows the same argumentation.

\begin{lem}\label{mixed-space-ineq}
 Let $u \in L_2(\Omega)$, and assume that $\pl_1 u \in L_2(\Omega)$. Then
 \[
\|u(x_1,\cdot)\|_{L_2(0,\pi)} \leqslant  \|u\|_2^{1/2} (\|u\|_2^{1/2} + \sqrt{2}\|\partial_1 u\|_2^{1/2}) \qquad (\text{a.e. } x_1 \in (-\pi,\pi)).
 \]
\end{lem}
\begin{proof}
We observe that $((-\pi,\pi) \ni x_1 \mapsto \|u(x_1,\cdot)\|_{L_2(0,\pi)}^2 \in \R) \in H^1(-\pi,\pi) \subseteq W_1^1(-\pi,\pi)$ with 
\[
 \frac{d}{dx_1}\|u(x_1,\cdot)\|_{L_2(0,\pi)}^2 = 2\int_0^\pi u(x_1,x_2)\pl_1u(x_1,x_2) \,dx_2 
 \]
 for a.e. $x_1 \in (-\pi,\pi)$. Now, using the fact that 
 \[
 \|f\|_{L_{\infty}(-\pi,\pi)} \leqslant \int_{-\pi}^\pi |f| \dx_1 + \int_{-\pi}^\pi |f'| \,dx_1 \qquad (f \in W_1^1(-\pi,\pi)), 
 \]
we have for a.e. $x_1' \in (-\pi,\pi)$ that 
\begin{align*}
&\|u(x_1',\cdot)\|_{L_2(0,\pi)}^2 \leqslant \int_{-\pi}^\pi \|u(x_1,\cdot)\|_{L_2(0,\pi)}^2 \,dx_1 + \int_{-\pi}^\pi \bigg|\frac{d}{dx_1}\|u(x_1,\cdot)\|_{L_2(0,\pi)}^2 
\bigg|\,dx_1  \\
&\leqslant \|u\|_2^2 + 2\int_{\Omega} |u||\pl_1 u| \dx \leqslant  \|u\|_2(\|u\|_2 + 2\|\pl_1 u\|_2). \tag*{\qedhere}
\end{align*}
\end{proof}
With all the above preparations we are now able to prove  Lemma~\ref{nonlin-eq-ineq}. 

\begin{proof}[Proof of Lemma~\ref{nonlin-eq-ineq}]
In the following proofs of (a)-(c) we assume that $u,v,w \in \lin \{e_{\ell}; \ \ell \in \mathcal{Z}\}$ and show the asserted equations and inequalities for these functions; 
the corresponding equations and inequalities then immediately follow by approximation. (Here one uses parts (a) and (b) for the approximation in part (c)).

(a) By H\"older's inequality, \eqref{Lady-ineq} and \eqref{poin-1} we have
\begin{align*}
 \int_{\Omega} |u (\pl_1 v)w| \dx  &\leqslant \|u\|_4 \|\pl_1v\|_2 \|w\|_4 \leqslant c_{\textnormal L}^2 \|u\|_2^{1/2}\|\nabla u\|_2^{1/2} \|\pl_1v\|_2 \|w\|_2^{1/2} \|\nabla w\|_2^{1/2} \\ 
&\leqslant c_{\textnormal L}^2\|\nabla u\|_2 \|\nabla v\|_2\|\nabla w\|_2.   
\end{align*}
Furthermore, with $f(x_1):= \|w(x_1,\cdot)\|_{L_2(0,\pi)} \ (\text{a.e. } x_1 \in (-\pi,\pi))$, it follows from \eqref{v-linfty-ineq} and Lemma~\ref{mixed-space-ineq} that
\begin{align*}
\int_{\Omega} |v_u(\pl_2 v) w| \dx  &\leqslant \int_{-\pi}^\pi \|v_u(x_1,\cdot)\|_{L_{\infty}(0,\pi)} \|\pl_2v(x_1,\cdot)\|_{L_2(0,\pi)} \|w(x_1,\cdot)\|_{L_2(0,\pi)} \,dx_1 \\
&\leqslant \sqrt{\pi} \int_{-\pi}^\pi 
\|\pl_1 u(x_1,\cdot)\|_{L_2(0,\pi)} \|\pl_2v(x_1,\cdot)\|_{L_2(0,\pi)} f(x_1) \,dx_1 \\
&\leqslant \sqrt{\pi} \|\pl_1u\|_2\|\pl_2v\|_2 \|f \|_{L_{\infty}(0,\pi)} \\
&\leqslant \sqrt{\pi}\|\pl_1u\|_2\|\pl_2v\|_2 \|w\|_2^{1/2}(\|w\|_2^{1/2} + \sqrt{2}\|\pl_1 w\|_2^{1/2}) \\
&\leqslant \sqrt{\pi}(1 + \sqrt{2})\|\nabla u\|_2 \|\nabla v\|_2 \|\nabla w\|_2.
\end{align*}

(b) Thanks to \eqref{Lady-ineq}, \eqref{poin-1} and \eqref{eq-Delta}, we obtain 
\begin{align*}
\int_{\Omega} |u (\pl_1v) w| \dx \leqslant \|u\|_4 \|\pl_1v\|_4 \|w\|_2 &\leqslant c_{\textnormal L}^2\|u\|_2^{1/2}\|\nabla u\|_2^{1/2}\|\pl_1v\|_2^{1/2}\|\pl_1\nabla v\|_2^{1/2} \|w\|_2 \\
&\leqslant c_{\textnormal L}^2 \|\nabla u\|_2\|\nabla v\|_2^{1/2}\|\Delta v\|_2^{1/2} \|w\|_2.
\end{align*}
With $f(x_1):= \|\pl_2v(x_1,\cdot)\|_{L_2(0,\pi)} \ (\text{a.e } x_1 \in (-\pi,\pi))$, and by \eqref{v-linfty-ineq}, Lemma~\ref{mixed-space-ineq}, \eqref{poin-2}, \eqref{poin-1} and \eqref{eq-Delta}, we further have
\begin{align*}
\int_{\Omega} |v_u (\pl_2 v)w| \dx  &\leqslant \sqrt{\pi}\int_{-\pi}^\pi \|\pl_1u(x_1,\cdot)\|_{L_2(0,\pi)} f(x_1) \|w(x_1,\cdot)\|_{L_2(0,\pi)} \,dx_1 \\
&\leqslant \sqrt{\pi}\|\pl_1u\|_2 \|f\|_{L_{\infty}(0,\pi)} \|w\|_2 \\
&\leqslant \sqrt{\pi} \|\pl_1u\|_2 \| \pl_2v\|_2^{1/2}(\|\pl_2v\|_2^{1/2} + \sqrt{2}\|\pl_1\pl_2 v\|_2^{1/2})\|w\|_2 \\
&\leqslant \sqrt{\pi} \|\pl_1u\|_2 \|\pl_2v\|_2^{1/2}(\|\pl_2v\|_2^{1/2} + \sqrt{2}\|\pl_1\nabla v\|_2^{1/2})\|w\|_2 \\
&\leqslant \sqrt{\pi}(1 + \sqrt{2})  \|\nabla u\|_2 \|\nabla v\|_2^{1/2}\|\Delta v\|_2^{1/2}\|w\|_2 .
\end{align*}

(c) Since $|v|^2 \in \hhp$, we can use integration by parts and \eqref{int-parts-2} to obtain 
\begin{align*}
\int_{\Omega} B(u,v)v  \dx &= \int_{\Omega} u (\pl_1 v) v + v_{u} (\pl_2 v)v \dx \\ 
&=  \frac{1}{2} \int_{\Omega} u \pl_1(|v|^2) + v_{u}\pl_2(|v|^2) \dx = 0.
\end{align*}
Moreover, since $|\pl_2u|^2 \in \hhp$, it again follows from integration by parts and \eqref{int-parts-2} that 
\begin{align*}
\langle B(u,u),\pl_2^2u \rangle_2 &= \int_{\Omega} u\pl_1u \pl_2^2u + v_{u} \pl_2 u \pl_2^2u \dx \\
&= - \int_{\Omega} |\pl_2u|^2\pl_1u + u\pl_1\pl_2u \pl_2 u - \frac{1}{2}v_{u} \pl_2 |\pl_2u|^2 \dx \\
&= - \int_{\Omega} |\pl_2u|^2\pl_1u + \frac{1}{2} u \pl_1|\pl_2u|^2 - \frac{1}{2}\pl_1u |\pl_2u|^2 \dx = 0.
\end{align*}
This completes the proof.
\end{proof}

\section{Existence and uniqueness of global strong solutions}\label{sec-strsol}

In this section we prove Theorem~\ref{main-thm}, our first main result. Along with this proof we obtain  useful equalities and inequalities for the proof of Theorem \ref{main-thm-attrac}. 

\begin{proof}[Proof of Theorem \ref{main-thm}]
The existence of global weak and strong solutions of \eqref{evol-1} is proven by the usual Galerkin approximation procedure, and by passing to the limit using an appropriate compactness theorem of Aubin-Lions (cf. \cite[Lemma~7.7]{Rou00}), after suitable a-priori estimates have been established. To do so, fix $n \in \N$, let $(e_j)_{j \in \N}$ be a counting of the orthonormalbasis $(e_{\ell})_{\ell \in \mathcal{Z}}$ of $\ldiv$ given by \eqref{ortho-1}-\eqref{ortho-2}, $V_n:= \lin \{e_1,\dots,e_n \}$ and let $P_n$ be the orthogonal projection from $V_n$ along $\ldiv$. Let further $u_n$ be the unique solution of \eqref{evol-1} obtained in the $n$-th approximation step of the Galerkin method with respect to the function spaces $V_n$, i.e., $u_n \in H^1(0,T;V_n)$ such that $u_n(0) = \sum_{j=1}^n \langle u_0,e_j \rangle_2 e_j$ and 
\begin{align}
&\langle u_n'(t),v \rangle_2 + \mu\int_{\Omega}\nabla u_n(t) \cdot \nabla v \dx + \langle u_n(t)\pl_1u_n(t) +v_{u_n(t)}\pl_2u_n(t), v \rangle_2 \tag*{} \\ 
&= \alpha\langle u_n(t),v \rangle_2 + \beta \langle v_{u_n(t)},v \rangle_2 + \langle K(t),v \rangle_{H_{\textnormal{div}}^{-1}(\Omega),\hhpu} \qquad (v \in V_n) \label{Gal-approx-eq} 
\end{align}
for a.e. $t \in (0,T)$; in particular, $\|u_n(0)\|_2 \leqslant \|u_0\|_2$ and $u_n(t) \in \chpu$ for a.e. $t \in (0,T)$, because of the fact that $V_n \sq \chpu$; see paragraph above \eqref{poin-1} for this fact. 
In a first step we will now show an estimate for $\sup_{t \in [0,T]}\|u_n(t)\|_2$. For simplicity we omit the index $n$ from now on.

\medskip
(i) \textbf{$\|u(t)\|_2$-estimate}: Putting $v = u(t)$ in~\eqref{Gal-approx-eq} and using Lemma~\ref{nonlin-eq-ineq}(c), we obtain 
\begin{equation}\label{attrac-eq-1}
\frac{1}{2}\frac{d}{dt}\|u(t)\|_2^2 + \mu \|\nabla u(t)\|_2^2 = \alpha \|u(t)\|_2^2 + \beta \langle v_{u(t)},u(t)\rangle_2 + \langle K(t),u(t) \rangle_{H_{\textnormal{div}}^{-1}(\Omega),\hhpu}
\end{equation}
for a.e. $t \in (0,T)$. Hence, using the Cauchy-Schwarz inequality and \eqref{v-l2-ineq} 
and then the Peter-Paul inequality (i.e. $ab \leqslant \frac{1}{2}(\gamma a^2 + \gamma^{-1}b^2)$ for all $a,b \in \R$, $\gamma > 0$), 
we find some $c_1 \geqslant 0$ only depending on $\alpha$, $\beta$ and $\mu$ such that 
\begin{align}
\frac{1}{2} \frac{d}{dt} \|u(t)\|_2^2 + \mu \|\nabla u(t)\|_2^2 &\leqslant \alpha  \|u(t)\|_2^2 + \pi|\beta|\|\pl_1u(t)\|_2\|u(t)\|_2 + \|K(t)\|_{H_{\textnormal{div}}^{-1}(\Omega)}\|\nabla u(t)\|_2 \label{e:attr-est} \\
&\leqslant c_1\|u(t)\|_2^2 + c_1\|K(t)\|_{H_{\textnormal{div}}^{-1}(\Omega)}^2 + \frac{\mu}{2}\|\nabla u(t)\|_2^2 \label{e:attr-est2}
\end{align}
for a.e. $t \in (0,T)$. Subtracting this equation by $\frac{\mu}{2}\|\nabla u(t)\|_2^2$, multiplying by $2$ and integrating the resulting inequality from $0$ to $t$ gives
\begin{align}
&\|u(t)\|_2^2-\|u(0)\|_2^2 + \mu \int_0^t \|\nabla u(s)\|_2^2 \,ds  \tag*{} \\ 
&\leqslant 2c_1 \int_0^t\|u(s)\|_2^2 \,ds + 2c_1\int_0^T \|K(s)\|_{H_{\textnormal{div}}^{-1}(\Omega)}^2 \,ds \qquad (t \in [0,T]). \label{l2-ineq-1} 
\end{align}
Now, Gronwall's inequality, the fact that $\|u(0)\|_2 \leqslant \|u_0\|_2$ and the assumption imply that
\begin{equation}\label{u-gronwall}
\|u(t)\|_2^2 \leqslant (\|u_0\|_2^2 + 2c_1\|K\|_{L_2(0,T;H_{\textnormal{div}}^{-1}(\Omega))}^2)e^{2c_1t}\leqslant (1+2c_1)r^2e^{2c_1T}=:C_1;   
\end{equation}
note that $C_1$ only depends on $\alpha$, $\beta$, $\mu$, $T$ and $r$. Together with~\eqref{l2-ineq-1} it follows that
\begin{equation}\label{h1-ineq-1}
\int_0^T \|\nabla u(s)\|_2^2 \,ds \leqslant \mu^{-1}(\|u_0\|_2^2 + 2c_1\|K\|_{L_2(0,T;H_{\textnormal{div}}^{-1}(\Omega))}^2)e^{2c_1T}=\mu^{-1}C_1;
\end{equation}
Thus, {in order} to obtain the existence of weak solutions, i.e., the first assertion of Theorem~\ref{main-thm}, it remains to show an a-priori estimate for $u'$ in $L_1(0,T;H_{\textnormal{div}}^{-1}(\Omega))$. 

\medskip
(ii) \textbf{$\|u'\|_{L_1(0,T;H_{\textnormal{div}}^{-1}(\Omega))}$-estimate}: {Note the estimate $\|\nabla Pv\|_2 = \|Pv\|_{1,\textnormal{div}} \leqslant \|v\|_{1,\textnormal{div}} \ (v \in \hhpu)$, which is due to the fact that $P|_{\hhpu}$ is an orthogonal projection onto $V$ along $(\hhpu, \|\cdot\|_{1,\textnormal{div}}$) and thus is a contraction with respect to the $\|\cdot\|_{\textnormal{div}}$-norm; for the definition of $\|\cdot\|_{\textnormal{div}}$ recall \eqref{def-norm-div}.} Then, using that $u'(t) \in V$ for the first equality, \eqref{Gal-approx-eq} for the first inequality and Lemma~\ref{nonlin-eq-ineq}(a), \eqref{poin-1} and \eqref{v-l2-ineq} for the second inequality, we obtain
\begin{align*}
|\langle u'(t),v \rangle_2| &= |\langle u'(t),Pv \rangle_2| \leqslant \mu\|\nabla u(t)\|_2 \|\nabla Pv\|_2 + |\langle B(u(t),u(t)),Pv \rangle_2| \\
&\mkern20mu + |\alpha|\|u(t)\|_2\|Pv\|_2 + |\beta|\|v_{u(t)}\|_2\|Pv\|_2 + \|K(t)\|_{\hhpui}\|\nabla Pv\|_2  \\ 
&\leqslant ((\mu + |\alpha| + \pi|\beta|)\|\nabla u(t)\|_2 + M_B\|\nabla u(t)\|_2^2 + \|K(t)\|_{\hhpui}){\|v\|_{1,\textnormal{div}} }
\end{align*}
for all $v \in \hhpu$, which in turn implies that
\begin{align*}
\|u'(t)\|_{\hhpui} \leqslant c_2\|\nabla u(t)\|_2 + M_B\|\nabla u(t)\|_2^2 + \|K(t)\|_{\hhpui} 
\end{align*}
for a.e. $t \in (0,T)$; here $c_2:= \mu + |\alpha| + \pi|\beta|$. Now, integrating this inequality from $0$ to $T$, using H\"older's inequality inequality for the first inequality, the Peter-Paul inequality for the second inequality and \eqref{h1-ineq-1} for the third inequality, we obtain
\begin{align*}
\|u'\|_{L_1(0,T;\hhpui)} &\leqslant c_2\sqrt{T}\|\nabla u\|_{L_2(0,T;L_2(\Omega))} + M_B\|\nabla u\|_{L_2(0,T;L_2(\Omega))}^2 + \sqrt{T}\|K\|_{L_2(0,T;\hhpui)}  \\ 
&\leqslant \bigg(M_B + \frac{c_2^2}{2}\bigg)\|\nabla u\|_{L_2(0,T;L_2(\Omega))}^2 + \frac{1}{2}\|K\|_{L_2(0,T;\hhpui)}^2 + T \\
&\leqslant \mu^{-1}\bigg(M_B + \frac{c_2^2}{2} \bigg)C_1 + \frac{1}{2}\|K\|_{L_2(0,T;\hhpui)}^2 + T. 
\end{align*}
Consequently, we {infer} %conclude
the existence of weak solutions.

\medskip
{Next,} we show the existence of strong solutions. For this, we assume that $u_0 \in \hhpu$ and $K \in L_2(0,T;\ldiv)$ until the end of this proof. Note that $\|\nabla u(0)\|_2 \leqslant \|\nabla u_0\|_2$; recall that here $u=u_n$ is the Galerkin approximation. 
Again we only establish necessary a-priori estimates, and start with proving an estimate for
$\sup_{t \in [0,T]}\|\pl_2 u(t)\|_2$.

\medskip
(iii) \textbf{$\|\pl_2u(t)\|_2$-estimate}: Observe that $\pl_2^2v \in V$ for all $v \in V$. Hence, $\pl_2^2 u(t) \in V$ for a.e. $t \in (0,T)$ since $u \in H^1(0,T;V)$. Now, by putting $v = - \pl_2^2u(t)$ in \eqref{Gal-approx-eq}, and using integration by parts, Lemma~\ref{nonlin-eq-ineq}(c) and \eqref{int-parts-2}, we arrive at
\begin{equation}\label{attrac-eq-2}
\frac{1}{2}\frac{d}{dt} \|\pl_2u(t)\|_2^2 + \mu \|\pl_2\nabla u(t)\|_2^2 = -\alpha \|\pl_2 u(t)\|_2^2 - \beta \langle \pl_1u(t),\pl_2u(t)\rangle_2 - \langle K(t),\pl_2^2u(t) \rangle_2
\end{equation}
for a.e. $t \in (0,T)$. Applying the Cauchy-Schwarz-inequality and the Peter-Paul inequality to the right-hand side of this equation, we find some $c_3 \geqslant 0$ only depending on $\alpha$, $\beta$ and $\mu$ such that
\[
\frac{1}{2} \frac{d}{dt} \|\pl_2u(t)\|_2^2 + \mu \|\pl_2\nabla u(t)\|_2^2 \leqslant c_3\|\nabla u(t)\|_2^2 + c_3\|K(t)\|_2^2 + \frac{\mu}{2}\|\pl_2\nabla u(t)\|_2^2 
\]
for a.e. $t \in (0,T)$. Thus, subtracting this inequality by $\frac{\mu}{2}\|\pl_2\nabla u(t)\|_2^2$, multiplying by $2$ and integrating the resulting inequality from $0$ to $t$ yields
\begin{align*}
&\|\pl_2u(t)\|_2^2-\|\pl_2u(0)\|_2^2 + \mu \int_0^t \|\pl_2\nabla u(s)\|_2^2\,ds \leqslant 2c_3\int_0^t \|\nabla u(s)\|_2^2 \,ds + 2c_3\int_0^T \|K(s)\|_2^2 \,ds
\end{align*}
for all $t \in [0,T]$. Therefore, by \eqref{h1-ineq-1}, the fact that $\|\pl_2 u(0)\|_2 \leqslant \|\nabla u(0)\|_2 \leqslant \|\nabla u_0\|_2{\leqslant r}$, with $r$ from the Theorem statement, we obtain
\begin{equation}\label{h1-ineq-3}
\|\pl_2 u(t)\|_2^2 + \mu \int_0^t \|\pl_2\nabla u(s)\|_2^2 \,ds \leqslant (1+2c_3)r^2 + 2c_3C_1 = :C_2 \qquad (t \in [0,T]);
\end{equation}
in particular,
\begin{equation}\label{h2-ineq-1}
\int_0^T \|\pl_2\nabla u(s)\|_2^2 \,ds \leqslant \mu^{-1}C_2. 
\end{equation}

\medskip
(iv) \textbf{$\|\pl_1u(t)\|_2$-estimate}: Note that $\pl_1^2u(t) \in V$ for a.e. $t \in (0,T)$. Choosing $v = -\pl_1^2u(t)$ in~\eqref{Gal-approx-eq},  
we find, {as explained below,} some $c_4 \geqslant 0$ only depending on $\alpha$, $\beta$, $\mu$ and $M_B'$ such that
\begin{align}
&\frac{1}{2}\frac{d}{dt} \|\pl_1u(t)\|_2^2 + \mu \|\pl_1\nabla u(t)\|_2^2 \nonumber \\
&\leqslant |\alpha| \|u(t)\|_2\|\pl_1^2u(t)\|_2 + \pi|\beta| \|\pl_1 u(t)\|_2 \|\pl_1^2u(t)\|_2 + \|K(t)\|_2\|\pl_1^2u(t)\|_2 \nonumber \\ 
&\mkern20mu + M_B'(\| u(t)\|_2^{1/2}\|\nabla u(t)\|_2^{1/2} \|\pl_1u(t)\|_2^{1/2}\|\pl_1\nabla u(t)\|_2^{1/2}  + \|\pl_1 u(t)\|_2\|\pl_2 u(t)\|_2^{1/2}(\|\pl_2u(t)\|_2^{1/2} \nonumber \\  &\mkern20mu + \|\pl_1 \nabla u(t)\|_2^{1/2}))\|\pl_1^2u(t)\|_2 \nonumber \\
&\leqslant c_4(\|u(t)\|_2^2 + \|\pl_1 u(t)\|_2^2 + \|K(t)\|_2^2 + \|u(t)\|_2^2\|\nabla u(t)\|_2^2\|\pl_1u(t)\|_2^2 \nonumber \\ 
&\mkern20mu + \|\pl_2u(t)\|_2^2\|\pl_1u(t)\|_2^2 + \|\pl_2u(t)\|_2^2\|\pl_1u(t)\|_2^4) + \frac{\mu}{2}\|\pl_1\nabla u(t)\|_2^2 \label{attrac-eq-3} 
\end{align}
for a.e. $t \in (0,T)$. 
{Here we used} the Cauchy-Schwarz inequality, \eqref{v-l2-ineq} and Lemma~\ref{nonlin-eq-ineq}(b) for the first inequality, and Young's inequality as well as the obvious fact that $\|\pl_1^2v\|_2 \leqslant \|\pl_1\nabla v\|_2 \ (v \in H^1(\Omega))$ for the second inequality.

Subtracting $\frac{\mu}{2}\|\pl_1\nabla u(t)\|_2^2$ {from} \eqref{attrac-eq-3}, multiplying by $2$ and integrating the resulting inequality from $0$ to $t$ gives
\begin{align}
&\|\pl_1u(t)\|_2^2 - \|\pl_1u(0)\|_2^2 + \mu\int_0^t \|\pl_1\nabla u(s)\|_2^2\,ds \nonumber \\  
&\leqslant 2c_4\int_0^T \|u(s)\|_2^2 \,ds +  2c_4\int_0^T \|K(s)\|_2^2 \,ds + 2c_4\int_0^t (1 + \|u(s)\|_2^2\|\nabla u(s)\|_2^2 \nonumber \\ 
&\mkern20mu + \|\pl_2u(s)\|_2^2 + \|\pl_2u(s)\|_2^2\|\pl_1u(s)\|_2^2)\|\pl_1u(s)\|_2^2 \,ds \qquad (t \in [0,T]).  \label{h1-ineq-4}
\end{align}
Hence, by Gronwall's inequality, 
\eqref{u-gronwall}, \eqref{h1-ineq-1}, \eqref{h1-ineq-3} and assumption, we obtain
\begin{align}
\|\pl_1u(t)\|_2^2 &\leqslant \biggl(\|\pl_1u(0)\|_2^2 + 2c_4\int_0^T \|u(s)\|_2^2 \,ds +  2c_4\int_0^T \|K(s)\|_2^2 \,ds \biggr) \nonumber \\
&\mkern20mu \exp\biggl( 2c_4\int_0^t 1 + \|u(s)\|_2^2\|\nabla u(s)\|_2^2 + \|\pl_2u(s)\|_2^2  + \|\pl_2u(s)\|_2^2\|\pl_1u(s)\|_2^2 \,ds \biggr) \nonumber \\ 
&\leqslant ((1+2c_4)r^2 + 2c_4C_1T)\exp(2c_4(T + \mu^{-1}C_1^2 + C_2T + C_1C_2))=:C_3 \label{h1-ineq-infty}
\end{align}
for all $t \in [0,T]$. Using this back in 
\eqref{h1-ineq-4} yields
\begin{equation}\label{h2-ineq-2}
\int_0^T \|\pl_1\nabla u(s)\|_2^2 \,ds \leqslant \mu^{-1}((1 + 2c_4)r^2 + 2c_4(T + \mu^{-1}C_1^2 + C_2T + C_1C_2)C_3)=:C_4.   
\end{equation}
We further infer from \eqref{h1-ineq-infty} and \eqref{h1-ineq-3} that
\begin{equation}\label{h1-ineq-infty-2}
\sup_{t \in [0,T]}\|\nabla u(t)\|_2^2 \leqslant C_3 + C_2:= C_5.     
\end{equation}
Due to \eqref{eq-Delta}, we further deduce from~\eqref{h2-ineq-1} and~\eqref{h2-ineq-2} that
\begin{align}\label{h2-ineq-3} 
\int_0^T \|\Delta u(s)\|_2^2\,ds \leqslant C_4 + \mu C_2=:C_6.
\end{align}

(v) \textbf{$\|u'(t)\|_2$-estimate}: 
Choosing $v = u'(t)$ in~\eqref{Gal-approx-eq}, it follows from integration by parts that 
\begin{align*}
\|u'(t)\|_2^2 &= \mu\langle \Delta u(t),u'(t) \rangle_2- \langle u(t)\pl_1 u(t) - v_{u(t)}\pl_2 u(t),u'(t) \rangle_2 + \alpha \langle u(t),u'(t) \rangle_2 \\ 
&\mkern20mu + \beta \langle v_{u(t)}, u'(t)\rangle_2 + \langle K(t),u'(t)\rangle_2 
\end{align*}
for a.e. $t \in (0,T)$. Now, using the Cauchy-Schwarz inequality, Lemma~\ref{nonlin-eq-ineq}(b) and \eqref{v-l2-ineq}, we have
\begin{align*}
\|u'(t)\|_2^2 &\leqslant (\mu\|\Delta u(t)\|_2 + M_B''\|\nabla u(t)\|_2\|\Delta u(t)\|_2 + |\alpha|\|u(t)\|_2 + \pi|\beta|\|\nabla u(t)\|_2 \\ 
&\mkern20mu + \|K(t)\|_2) \|u'(t)\|_2
\end{align*}
for a.e. $t \in (0,T)$. Hence, by the inequality $(a + b)^2 \leqslant 4(a^2 + b^2) \ (a,b \in \R)$, we find some $c_5 \geqslant 0$ only depending on $\alpha$, $\beta$, $\mu$ and $M_B''$ such that
\begin{align}
\|u'(t)\|_2^2 &\leqslant c_5(\|\Delta u(t)\|_2^2 +\|\nabla u(t)\|_2^2\|\Delta u(t)\|_2^2 +\|u(t)\|_2^2 + \|\nabla u(t)\|_2^2 +\|K(t)\|_2^2) \label{attrac-eq-4}
\end{align}
for a.e. $t \in (0,T)$. Therefore, integrating this inequality from $0$ to $T$ and using~\eqref{h2-ineq-3}, \eqref{h1-ineq-infty-2}, \eqref{u-gronwall}, \eqref{h1-ineq-1} and the assumption, we finally obtain
\begin{align*}
\int_0^T \|u'(s)\|_2^2\,ds &\leqslant c_5\int_0^T \|\Delta u(s)\|_2^2 + \|\Delta u(s)\|_2^2\sup_{t \in [0,T]}\|\nabla u(t)\|_2^2  \\ &\mkern20mu + \sup_{t \in [0,T]}\|u(t)\|_2^2 + \|\nabla u(s)\|_2^2 + \|K(s)\|_2^2 \,ds \\
&\leqslant c_5(C_6 + C_5C_6 + C_1T + \mu^{-1}C_1 +r^2).
\end{align*}
Thus, we have shown all necessary a-priori estimates to conclude the existence of a strong solution. {Uniqueness follows as a byproduct} from \eqref{contin-depend}, i.e., the continuous dependence on the initial data, which we prove next. 

\medskip
(vi) \textbf{Continuous dependence on $u_0$}: Let $u_{0,1},u_{0,2} \in \hhpu$, and let $u_1,u_2 \in L_2(0,T;\Hhpu) \cap H^1(0,T;\ldiv)$ be two strong solutions of \eqref{evol-1} on $(0,T)$ with respect to $u_{0,1}$ and $K$ as well as $u_{0,2}$ and $K$, respectively. Let $u:=u_1-u_2$. Then, subtracting  \eqref{evol-1} for $u_2$ {from} \eqref{evol-1} for $u_1$, multiplying the resulting equation by $-\Delta u(t)$ and integrating it over $\Omega$, we obtain{, after rearranging terms,}
\begin{align}
&\frac{1}{2} \frac{d}{dt}\|\nabla u(t)\|_2^2 + \mu \|\Delta u(t)\|_2^2 - \langle B(u_1(t),u_1(t)),\Delta u(t) \rangle_2 + \langle B(u_2(t),u_2(t)), \Delta u(t) \rangle_2 \nonumber \\
&= -\alpha \langle u(t),\Delta u(t) \rangle_2 + \beta \langle v_{u(t)}, \Delta u(t) \rangle_2 \qquad (\text{a.e. } t \in (0,T)). \label{uniq-eq-1}
\end{align}
Now observe that 
\begin{align}
&\langle B(u_1(t),u_1(t)),v \rangle_2 - \langle B(u_2(t),u_2(t)),v \rangle_2 \nonumber \\ 
&= \langle B(u(t),u_1(t)),v \rangle_2 + \langle B(u_2(t),u(t)),v \rangle_2 \qquad (v \in L_2(\Omega)) \label{uniq-eq-2}
\end{align}
for a.e. $t \in (0,T)$. Based on this, we estimate \eqref{uniq-eq-1} as explained below, with some $c_6 \geqslant 0$ only depending on $\alpha$, $\beta$, $\mu$ and $M_B''$, for a.e. $t \in (0,T)$: 
\begin{align}
&\frac{1}{2} \frac{d}{dt}\|\nabla u(t)\|_2^2 + \mu\|\Delta u(t)\|_2^2 
 \nonumber \\ &\leqslant |\alpha|\|u(t)\|_2\|\Delta u(t)\|_2 + \pi |\beta| \|\pl_1 u(t)\|_2 \|\Delta u(t)\|_2 \nonumber \\ &\mkern20mu + |\langle B(u(t),u_1(t)),\Delta u(t) \rangle_2| + |\langle B(u_2(t),u(t)),\Delta u(t) \rangle_2| \nonumber \\
&\leqslant |\alpha|\|\nabla u(t)\|_2\|\Delta u(t)\|_2 + \pi |\beta| \|\nabla u(t)\|_2 \|\Delta u(t)\|_2 \nonumber \\ &\mkern20mu + M_B''\|\nabla u(t)\|_2\|\Delta u_1(t)\|_2 \|\Delta u(t)\|_2 \nonumber \\ &\mkern20mu + M_B''\|\nabla u_2(t)\|_2\|\nabla u(t)\|_2^{1/2}\|\Delta u(t)\|_2^{3/4} \nonumber \\ &\leqslant c_6(1 + \|\Delta u_1(t)\|_2^2 + \|\nabla u_2(t)\|_2^4)\|\nabla u(t)\|_2^2 + \frac{\mu}{2} \|\Delta u(t)\|_2^2,\label{uniq-eq}
\end{align}
for a.e. $t \in (0,T)$. 
{Here we used} the Cauchy-Schwarz inequality, \eqref{v-l2-ineq} and \eqref{uniq-eq-2} for the first inequality, \eqref{poin-1} and Lemma~\ref{nonlin-eq-ineq}(b) for the second inequality, and Young's inequality for the third inequality. 

Subtracting $\frac{\mu}{2}\|\nabla u(t)\|_2^2$ from \eqref{uniq-eq}, multiplying by $2$ and integrating the resulting inequality from $0$ to $t$ gives  
\begin{align*}
&\|\nabla u(t)\|_2^2 - \|\nabla (u_{0,1} - u_{0,2})\|_2^2 + \mu\int_0^t\|\Delta u(s)\|_2^2 \,ds \\ &\leqslant 2c_6\int_0^t (1 + \|\Delta u_1(s)\|_2^2 + \|\nabla u_2(s)\|_2^4)\|\nabla u(s)\|_2^2 \,ds
\end{align*}
for all $t \in [0,T]$. Therefore, Gronwall's inequality implies that
\begin{equation}\label{uniq-ineq-1}
\|\nabla u(t)\|_2^2 \leqslant \|\nabla(u_{0,1} - u_{0,2})\|_2^2 \exp \biggl(2c_6(1 + \sup_{t \in [0,T]}\|\nabla u_2(t)\|_2^4)T + 2c_6\int_0^T\|\Delta u_1(s)\|_2^2 \,ds\biggr) 
\end{equation}
for all $t \in [0,T]$; in particular, if $u_{0,1} = u_{0,2}$, then $\nabla u(t) = 0$ for all $t \in [0,T]$, i.e., $u_1 = u_2$, and thus uniqueness follows. 

Let now $u_1$ and $u_2$ be the strong solutions of \eqref{evol-1} on $(0,T)$ with respect to $u_{0,1}$ and $K$ as well as $u_{0,2}$ and $K$, respectively, obtained in the steps (i)-(v). Then it follows from \eqref{uniq-ineq-1}, \eqref{h1-ineq-infty-2} and \eqref{h2-ineq-3} that
\begin{align*}
\|\nabla u(t)\|_2^2 &\leqslant C\|\nabla(u_{0,1} - u_{0,2})\|_2^2  \qquad (t \in [0,T]),
\end{align*}
with $C:= \exp(2c_6(1 + C_5^2)T + 2c_6C_6)$. Thus, we conclude that \eqref{contin-depend} holds, which in turn completes the proof of Theorem~\ref{main-thm}.
\end{proof}

Raising the regularity of the data, we can now show that the unique global strong solution obtained from Theorem~\ref{main-thm} becomes more regular. We need some equalites and inequalities obtained in the proof of this regularity result for the existence proof of a global attractor for \eqref{evol-1} later in Section 5.

\begin{thm}\label{main-thm-reg}
Let $T, r > 0$, $u_0 \in \Hhpu$ and $K \in H^1(0,T;\hhpui) \cap C([0,T];\ldiv)$ such that $\|\Delta u_0\|_2 \leqslant r$ and $\|K\|_{H^1(0,T;\hhpui)}, \|K\|_{C([0,T];\ldiv)} \leqslant r$. Let $u$ be the unique strong solution of \eqref{evol-1} on $(0,T)$ with respect to $u_0$ and $K$. Then $u \in L_{\infty}(0,T;\Hhpu) \cap H^1(0,T;\hhpu) \cap H^2(0,T;\hhpui)$.
\end{thm}
\begin{proof}
Let $u_n$, $V_n$, $P_n$, $c_1,\dots, c_6$ and $C_1, \dots,C_6$ be as in the proof of Theorem~\ref{main-thm}. To prove the assertion, we only establish necessary a-priori estimates for $u_n$ similar to the proof of Theorem~\ref{main-thm}. The assertion then follows by use of Aubin-Lions compactness theorem. For simplicity we omit the index $n$. To begin with, we observe that $u \in H^2(0,T;V)$ and 
\begin{align}
&\langle u''(t),v \rangle_2 + \mu \int_{\Omega} \nabla u(t) \cdot \nabla v \dx  + \langle B(u(t),u'(t)).v \rangle_2 + \langle B(u'(t),u(t)),v \rangle_2 \nonumber \\ 
&= \alpha \langle u'(t),v \rangle_2 + \beta\langle v_{u'(t)},v \rangle_2 + \langle K'(t),v \rangle_{\hhpui,\hhpu} \qquad (v \in V) \label{u-prime-ineq-1}
\end{align}
for a.e. $t \in (0,T)$, by \eqref{Gal-approx-eq} and the fact that $K \in H^1(0,T;\hhpui)$.  In a first step we will now show a-priori estimates for $\|u'(0)\|_2$ and $\sup_{t \in [0,T]}\|u'(t)\|_2$.

(i) \textbf{$\|u'(t)\|_2^2$-estimate}: Putting $t = 0$ in \eqref{attrac-eq-4} and using the assumption, we have that 
\begin{align}
\|u'(0)\|_2^2 &\leqslant c_5(\|\Delta u(0)\|_2^2 + \|\Delta u(0)\|_2^2 + \|u(0)\|_2^2 + \|\nabla u_0\|_2^2 + \|K(0)\|_2^2) \nonumber \\
&\leqslant c_5(r^4 + 4r^2)=: C_7. \label{u_0-prime-ineq-2} 
\end{align}
Further, by putting $v = u'(t)$ in \eqref{u-prime-ineq-1}, using \eqref{v-l2-ineq}, the Cauchy-Schwarz inequality as well as Lemma~\ref{nonlin-eq-ineq}(b) and (c), we obtain 
\begin{align*}
\frac{1}{2} \frac{d}{dt} \|u'(t)\|_2^2 + \mu \|\nabla u'(t)\|_2^2 &\leqslant |\alpha| \|u'(t)\|_2^2 +  \pi|\beta| \|\pl_1 u(t)\|_2^2 \|u'(t)\|_2 \\
&\mkern20mu + \|K'(t)\|_{\hhpui}\|\nabla u'(t)\|_2 + M_B''\|\nabla u'(t)\|_2 \|\Delta u(t)\|_2 \|u'(t)\|_2 
\end{align*}
for a.e. $t \in (0,T)$. By the Peter-Peter-Paul inequality, we find some $c_7 \geqslant 0$ only depending on $\mu$, $\alpha$, $\beta$ and $M_B''$ such that
\begin{equation}\label{u-prime-ineq-2}
\frac{1}{2} \frac{d}{dt} \|u'(t)\|_2^2 + \mu \|\nabla u'(t)\|_2^2 \leqslant c_7(1 + \|\Delta u(t)\|_2^2)\|u'(t)\|_2^2 + c_7\|K'(t)\|_{\hhpui}^2 + \frac{\mu}{2} \|\nabla u'(t)\|_2^2.
\end{equation}
for a.e. $t \in (0,T)$. Subtracting \eqref{u-prime-ineq-2} by $\frac{\mu}{2}\|\nabla u'(t)\|_2^2$, multiplying by 2 and integrating the resulting equation from $0$ to $t$ gives
\begin{align}
&\|u'(t)\|_2^2 - \|u'(0)\|_2^2 + \mu \int_0^t \|\nabla u'(s)\|_2^2 \,ds \nonumber \\ 
&\leqslant 2c_7 \int_0^t(1 + \|\Delta u(s)\|_2^2) \|u'(s)\|_2^2 \,ds
+ 2c_7 \int_0^t \|K'(s)\|_{\hhpui}^2 \,ds \label{u-prime-ineq-3}
\end{align}
for all $t \in [0,T]$. Now, Gronwall's inequality, \eqref{u_0-prime-ineq-2}, \eqref{h2-ineq-3} and the assumption imply that 
\begin{equation}\label{u-prime-ineq}
\|u'(t)\|_2^2 \leqslant (C_7 + 2c_7r^2)e^{2c_7(T + C_6)}=:C_8 \qquad (t \in [0,T]), 
\end{equation}
which together with \eqref{u-prime-ineq-3} yields 
\begin{equation}\label{u-prime-ineq-4}
\int_0^T \|\nabla u'(s)\|_2^2 \,ds \leqslant \mu^{-1}C_8.   
\end{equation}

(ii) \textbf{$\|\Delta u(t)\|_2^2$-estimate}: Putting $v = -\Delta u(t)$ in \eqref{Gal-approx-eq}, and using \eqref{v-l2-ineq} as well as Lemma~\ref{nonlin-eq-ineq}(b), we obtain
\begin{align*}
\mu \|\Delta u(t)\|_2^2 &\leqslant (\|u'(t)\|_2 + |\alpha|\|u(t)\|_2 + \pi |\beta| \|\nabla u(t)\|_2 + \|K(t)\|_2 \\ &\mkern20mu + M_B''\|\nabla u(t)\|_2^{3/2}\|\Delta u(t)\|_2^{1/2})\|\Delta u(t)\|_2 \qquad (t \in [0,T]).
\end{align*}
Thus, by the Peter-Paul inequality and the inequality $(a + b)^2 \leqslant 4(a^2 + b^2) \ (a,b \in \R)$, we find a $c_{8} \geqslant 0$ only depending on $\mu$, $\alpha$, $\beta$ and $M_B''$ such that
\begin{equation}\label{u-prime-ineq-5}
\|\Delta u(t)\|_2^2 \leqslant c_8(\|u'(t)\|_2^2 + \|u(t)\|_2^2 + \|\nabla u(t)\|_2^2 + \|\nabla u(t)\|_2^6 + \|K(t)\|_2^2) \qquad (t \in [0,T]).
\end{equation}
By \eqref{u-prime-ineq}, \eqref{u-gronwall} and \eqref{h1-ineq-infty-2} and assumption it follows that 
\[
\sup_{t \in [0,T]} \|\Delta u(t)\|_2^2 \leqslant c_8 (C_8 + C_1 + C_5 + (C_5)^3 + r^2).
\]
Consequently, it suffices to show a $\|u''(t)\|_{L_2(0,T;\hhpui)}$-estimate.

(iii) \textbf{$\|u''(t)\|_{L_2(0,T;\hhpui)}$-estimate}: 
As explained below, we claim that 
\begin{align*}
|\langle u''(t),v \rangle_2| = |\langle u''(t),Pv \rangle_2| &\leqslant \mu \|\nabla u'(t)\|_2 \|\nabla Pv\|_2 + |\langle B(u'(t),u(t)),Pv\rangle_2| \\ &\mkern20mu + |\langle B(u(t),u'(t)),Pv \rangle_2| + |\alpha|\|u'(t)\|_2\|Pv\| + \\ &\mkern20mu +\pi|\beta|\|\pl_1 u'(t)\|_2\|Pv\|_2 + \|K'(t)\|_{\hhpui} \|\nabla Pv\|_2 \\
&\leqslant ((\mu + |\alpha| + |\beta|)\|\nabla u'(t)\|_2 + 2M_B\|\nabla u'(t)\|_2\|\nabla u(t)\|_2 \\ &\mkern20mu 
+\|K'(t)\|_{\hhpui})\|\nabla v\|_2 \qquad (v \in \hhpu)
\end{align*}
for a.e. $t \in (0,T)$. 
Here, for the first equality we used that  $u''(t) \in V$, for the first inequality we used \eqref{u-prime-ineq-1} and \eqref{v-l2-ineq}, and for the second inequality we used \eqref{poin-1}, Lemma~\ref{nonlin-eq-ineq}(a) as well as the fact that $\|\nabla Pv\|_2 \leqslant \|v\|_{1,\textnormal{div}} \ (v \in \hhpu)$. 
From this we further infer that 
\[
\|u''(t)\|_{\hhpui}^2 \leqslant c_9 (\|\nabla u'(t)\|_2^2 + \|\nabla u'(t)\|_2^2\|\nabla u(t)\|_2^2 + \|K'(t)\|_{\hhpui}^2) 
\]
for a.e. $t \in (0,T)$, with some $c_9 \geqslant 0$ only depending on $\mu$, $\alpha$, $\beta$ and $M_B$. Now, integrating this inequality from $0$ to $T$, using \eqref{u-prime-ineq-4}, \eqref{h1-ineq-infty-2} and the assumption, we obtain
\[
\|u''\|_{L_2(0,T;\hhpui)}^2 \leqslant c_9(\mu^{-1}C_8 + \mu^{-1}C_5C_8 + r^2).
\]
This completes the proof of Theorem~\ref{main-thm-reg}.
\end{proof}

\section{Existence of compact global attractor}\label{sec-attractor}

In this section we show our second main result Theorem~\ref{main-thm-attrac}, i.e., the dynamics associated with \eqref{evol-1} feature a (compact) global attractor. 
For this, let $K \in L_{2,\textnormal{loc}}(0,\infty;\ldiv)$ be arbitrary but fixed throughout and $S\colon[0,\infty) \to M(\hhpu)$ from \eqref{e:semigr}.

Following \cite[Def.~I.1.2]{Tem97}, we say that a (compact) set $A \sq \hhpu$ is a (\textbf{compact}) \textbf{global attractor} of $S$ if
\begin{enumerate}
    \item[(i)] $A$ is invariant under $S$, i.e., $S(t)A = A$ for all $t \geqslant 0$, and
    \item[(ii)] $A$ attracts the bounded sets of $\hhpu$, i.e., for all bounded $B \sq \hhpu$ holds
    \[
    d(S(t)B,A) \to 0 \qquad (t \to \infty),
    \]
    where $d(B_1,B_2)$ is defined as the semidistance of two sets $B_1,B_2 \sq \hhpu$, i.e., 
    \[
    d(B_1,B_2):= \sup_{u \in B_1}  \inf_{v \in B_2} \|u-v\|_{H^1(\Omega)}.
    \]
\end{enumerate}

We use the following standard result for showing the existence of an attractor; for the formulation of this result we simplify  \cite[Thm.~I.1.1]{Tem97}, adapted to our situation. 
\begin{thm}\label{thm-attractor}
Let $H$ be a metric space, and $U\colon [0,\infty) \to M(H)$ a continuous semigroup. Let $B \sq H$ an \textbf{absorbing set} of $U$, i.e, for any $B_0 \sq H$ there exists a $t_0 \geqslant 0$ such that $U(t)B_0 \sq B$. Assume that there exists $t_1 > 0$ such that $U(t_1)B$ is compact. Then the \textbf{$\omega$-limit set} 
\[
A:= \bigcap_{s \geqslant 0} \overline{\bigcup_{t \geqslant s} U(t)B}
\]
of $B$ is a compact global attractor of $U$. Moreover, $A$ is the maximal bounded attractor of $U$ (for the inclusion relation).
\end{thm}

In order to apply Theorem~\ref{thm-attractor} in our case, we next 
show that $S$ has an absorbing set.  

\begin{prop}\label{absorb-ball}
Let $K \in H^1(0,\infty;\hhpui) \cap C_b([0,\infty);\ldiv)$, and assume that $\pi|\beta| < \mu$ if $\alpha \leqslant 0$ and $\alpha + \pi |\beta| < \mu$ if $\alpha > 0$. Then there exists an $R > 0$ such that $B_{\Hhpu}(0,R)$ is an absorbing set for $S$ in $\hhpu$, i.e., for any $r > 0$ there exists a $t_0 > 0$ such that $S(t)B_{\hhpu}(0,r) \sq B_{\Hhpu}(0,R)$ for all $t \geqslant t_0$.
\end{prop}
For the proof of Proposition~\ref{absorb-ball} we use the well-known \textit{uniform Gronwall Lemma}, which we state for the sake of completeness; 
see \cite[Lem.~III.1.1]{Tem97}.
\begin{lem}[The uniform Gronwall Lemma]\label{uniform-gronwall}
Let $t_0, c_1,c_2,c_3 \geqslant 0$, $r > 0$, $0 \leqslant a,b \in L_{1,\textnormal{loc}}(t_0,\infty)$ and $0 \leqslant f \in W_{1,\textnormal{loc}}^1(t_0,\infty)$, and assume that
\[
f'(t) \leqslant a(t)f(t) + b(t) \qquad (\text{a.e. } t \in (t_0,\infty)),
\]
and 
\[
\int_t^{t + r} a(s) \,ds \leqslant c_1, \qquad \int_t^{t + r} b(s) \,ds \leqslant c_2, \qquad \int_t^{t+r} f(s)\,ds \leqslant c_3 \qquad (t \geqslant t_0).
\]
Then
\[
f(t + r) \leqslant (c_2 + r^{-1}c_3)e^{c_1} \qquad (\text{a.e. } t \in (t_0,\infty)).
\]
\end{lem}

Now we can turn to the proof of Proposition~\ref{absorb-ball}.

\begin{proof}[Proof of Proposition~\ref{absorb-ball}]
Fix $r > 0$ and $u_0 \in B_{\hhpu}(0,r)$. Let $u_n$ and $V_n$ be as in the proofs of Theorem~\ref{main-thm} with respect to $u_0$ and $K$ for $T = \infty$. (Here we may choose $T = \infty$ due to the uniqueness of $u_n$ on intervals $(0,t)$ for arbitrary $t > 0$. In this case, we have that $u \in H_{\textnormal{loc}}^2(0,\infty;V)$; see Theorem~\ref{main-thm}.) Let further $c_1,\dots, c_9$ be as in the proofs of Theorem~\ref{main-thm} and Theorem~\ref{main-thm-reg}. As in the proof of Theorem~\ref{main-thm}, we show the assertion for $u_n$, after which the general assertion follows by approximation. For simplicity, we again omit the index $n$. To prove the assertion, we follow the standard method, which uses the uniform Gronwall Lemma. In a first step, we will now prove global bounds for $\|u(t)\|_2$ and $\int_t^{t+1} \|\nabla u(s)\|_2^2 \,ds$ independent of $n$ and $r$.

(i) \textbf{$\int_t^{t+1} \|\nabla u(s)\|_2^2 \,ds$-estimate}: By using \eqref{v-l2-ineq}, \eqref{poin-1}, the fact that $\|v\|_{\hhpui} \leqslant \|v\|_2 \ (v \in L_2(\Omega))$ and the Peter-Paul inequality, it follows from \eqref{attrac-eq-1} that
\begin{equation}\label{u-attrac-1}
\frac{d}{dt}\|u(t)\|_2^2 + \omega \|\nabla u(t)\|_2^2 \leqslant c_{10} \|K(t)\|_2^2 \qquad (t > 0)
\end{equation}
for some $c_{10} \geqslant 0$ only depending on $\alpha$, $\beta$ and $\mu$ and with $\omega:= \mu - \pi|\beta|$ if $\alpha \leqslant 0$ and $\omega:= \mu - \alpha - \pi|\beta|$ if $\alpha > 0$. Multiplying \eqref{u-attrac-1} by $e^{\omega t}$ and using \eqref{poin-1} further yields
\[
\frac{d}{dt}(e^{\omega t} \|u(t)\|_2^2) \leqslant c_{10} e^{\omega t} \|K(t)\|_2^2 \qquad (t > 0).
\]
Integrating this inequality from $0$ to $t$ and subsequent multiplication by $e^{-\omega t}$ gives
\[
\|u(t)\|_2^2 \leqslant e^{-\omega t}\|u_0\|_2^2 + c_{10} \int_0^t e^{\omega(s-t)} \|K(s)\|_2^2 \,ds \leqslant r^2e^{-\omega t} + c_{10}\|K\|_{L_2(0,\infty;\ldiv)}^2
\]
for all $t \geqslant 0$. Thus, there exists a $t_0 \geqslant 0$ such that
\begin{equation}\label{u-attrac-2}
\|u(t)\|_2^2 \leqslant 1 + c_{10}\|K\|_{L_2(0,\infty;\ldiv)}^2 =:M_1 \qquad (t \geqslant t_0);
\end{equation}
note that $M_1$ is independent of $n$ and $r$. Therefore, integrating \eqref{u-attrac-1} from $t$ to $t + 1$, we get that
\begin{align}
\int_t^{t + 1} \|\nabla u(s)\|_2^2\,ds &\leqslant \omega^{-1}(\|u(t)\|_2^2 + c_{10}\|K\|_{L_2(0,\infty;\ldiv)}^2) \nonumber \\
&\leqslant \omega^{-1}(M_1 + c_{10}\|K\|_{L_2(0,\infty;\ldiv)}^2)=:M_2 \qquad (t \geqslant t_0). \label{u-attrac-3}
\end{align}

(ii) \textbf{$\int_t^{t+1} \|\nabla \pl_2 u(s)\|_2^2 \,ds$-estimate}: Treating equation \eqref{attrac-eq-2} similar as \eqref{attrac-eq-1} in step (i), we obtain the inequality
\begin{equation}\label{u-attrac-4}
\frac{d}{dt} \|\pl_2u(t)\|_2^2 + \omega \|\nabla \pl_2 u(t)\|_2^2 \leqslant c_{10} \|K(t)\|_2^2 \qquad (t > 0).
\end{equation}
Now, by \eqref{u-attrac-3} and Lemma~\ref{uniform-gronwall}, it follows that 
\begin{equation}\label{u-attrac-5}
\|\pl_2 u(t)\|_2^2 \leqslant M_2  + c_{10}\|K\|_{L_2(0,\infty;L_{2,0}(\Omega))}^2=:M_3 \qquad (t \geqslant t_0 + 1). 
\end{equation}
Thus, integrating \eqref{u-attrac-4} from $t$ to $t+1$ yields
\begin{equation}\label{u-attrac-6}
\int_t^{t+1} \|\nabla \pl_2 u(s)\|_2^2 \,ds \leqslant \omega^{-1}(M_3  + c_{10}\|K\|_{L_2(0,\infty;\ldiv)}^2)=:M_4 \qquad (t \geqslant t_0 + 1). 
\end{equation}

(iii) \textbf{$\int_t^{t+1} \|\nabla \pl_1 u(s)\|_2^2 \,ds$-estimate}: Subtracting  $\frac{\mu}{2}\|\nabla \pl_1 u(t)\|_2^2$ from the last inequality in \eqref{attrac-eq-3}, multiplying by $2$ and using \eqref{poin-1}, \eqref{u-attrac-2} as well as \eqref{u-attrac-5}, we get the inequality 
\begin{align}
\frac{d}{dt} \|\pl_1u(t)\|_2^2 + \mu\|\nabla \pl_1u(t)\|_2^2 &\leqslant 2c_3(1 + (1 + M_1 + M_3)\|\nabla u(t)\|_2^2)\|\pl_1 u(t)\|_2^2 \nonumber \\ &\mkern20mu + 2c_3M_1 + 2c_3 \|K(t)\|_2^2 \qquad (t \geqslant t_0 + 1). \label{u-attrac-7}
\end{align}
Thus, due to \eqref{u-attrac-4}, Lemma~\ref{uniform-gronwall} implies that 
\begin{equation}\label{u-attrac-8}
\|\pl_1u(t)\|_2^2 \leqslant (M_2 + 2c_3(M_1 + \|K\|_{L_2(0,\infty;\ldiv)}^2))e^{2c_3(1 + (1 + M_1 + M_3)M_2)}=: M_5  
\end{equation}
for all $t \geqslant t_0 + 2$. Moreover, integrating \eqref{u-attrac-7} from $t$ to $t + 1$, we have that 
\begin{align}
\int_t^{t + 1} \|\nabla \pl_1 u(s)\|_2^2 \,ds &\leqslant \mu^{-1}(M_5 + 2c_3M_5(1 + (1 + M_1 + M_3)M_2) \nonumber \\ 
&\mkern20mu + 2c_3(N_1 + \|K\|_{L_2(0,\infty;\ldiv)}^2)) =: M_6 \qquad (t \geqslant t_0 + 2). \label{u-attrac-9} 
\end{align}

(iv) \textbf{$\int_t^{t+1} \|u'(s)\|_2^2 \,ds$-estimate}: By \eqref{u-attrac-5}, \eqref{u-attrac-8}, \eqref{eq-Delta}, \eqref{u-attrac-6} and \eqref{u-attrac-9} we obtain 
\begin{equation}\label{u-attrac-10}
\|\nabla u(t)\|_2^2 \leqslant M_3 + M_5 =: M_7, \quad \int_t^{t + 1} \|\Delta u(s)\|_2^2 \,ds \leqslant M_4 + M_6 =: M_8     
\end{equation}
for all $t \geqslant t_0 + 2$. Therefore, after integrating \eqref{attrac-eq-4} from $t$ to $t + 1$, we get that
\begin{equation}\label{u-attrac-11}
\int_t^{t + 1} \|u'(s)\|_2^2 \,ds \leqslant c_5(M_8 + M_7M_8 + M_1 + M_7 + \|K\|_{L_2(0,\infty;\ldiv)}^2) =:M_9 
\end{equation}
for all $t \geqslant t_0 + 2$.

(v) \textbf{$\|\Delta u(t)\|_2$-estimate}: Subtracting \eqref{u-prime-ineq-2} by $\frac{\mu}{2}\|\nabla u'(t)\|_2^2$ and multiplying by $2$, we have that
\[
\frac{d}{dt} \|u'(t)\|_2^2 + \mu\|\nabla u'(t)\|_2^2 \leqslant 2c_7(1 + \|\Delta u(t)\|_2^2)\|u'(t)\|_2^2 + 2c_7 \|K'(t)\|_{\hhpui}^2 
\]
for a.e. $t \in (0,\infty)$. Now, due to \eqref{u-attrac-10} and \eqref{u-attrac-11}, it follows from Lemma~\ref{uniform-gronwall} that
\begin{equation}\label{u-attrac-12}
\|u'(t)\|_2^2 \leqslant (M_9 + 2c_7\|K'\|_{L_2(0,\infty;\hhpui)}^2)e^{2c_7(1 + M_8)}=: M_{10} \qquad (t \geqslant t_0 + 3).
\end{equation}
Consequently, since by assumption $K \in C_b([0,\infty);\ldiv)$ and \eqref{u-attrac-12}, \eqref{u-attrac-5} as well as \eqref{u-attrac-10} hold, we deduce from \eqref{u-prime-ineq-5} that
\[
\|\Delta u(t)\|_2^2 \leqslant c_8(M_{10} + M_1 + M_7 + M_7^3 + \sup_{s \in [0,\infty)}\|K(s)\|_2^2) \qquad (t \geqslant t_0 + 3).
\]
This completes the proof since $c_8$, $M_{10}$, $M_1$ and $M_7$ do not depend on $n$ and $r$, which gives $R$ as claimed.
\end{proof}

Now that we have shown the existence of an absorbing set for $S$, we can prove Theorem~\ref{main-thm-attrac} by means of Theorem~\ref{thm-attractor}.

\begin{proof}[Proof of Theorem~\ref{main-thm-attrac}]
Recall that, for every $t > 0$, $S(t) \colon \hhpu \to \hhpu$ is continuous due to \eqref{contin-depend}. Therefore, since $S$ has an absorbing set in $\hhpu $ by Proposition~\ref{absorb-ball} and $\Hhpu$ is compactly embedded in $\hhpu$, the assertion follows from Theorem~\ref{thm-attractor}.
\end{proof}

\section{Discussion}

In \cite{CoJo22}, Constantin and Johnson  derived a model for nonlinear wave propagation in the troposphere that is particularly relevant to meteorological phenomena such as the 'morning glory' cloud formations. 
The slightly simplified form \eqref{main-sys}-\eqref{incom-cond} has been studied from a mathematical viewpoint concerning well-posedness in  \cite{MaRo23,AlGr24}, but global existence for general initial data remains unclear. 

In this paper, we introduced and analyzed a modification that is motivated by standard atmospheric modeling practices using Dirichlet boundary condition at the top of the troposphere. 
%Since $v$ represents the vertical velocity multiplied by air density, and both tend to vanish at high altitudes, 
%a commonly used boundary condition at the top is that $v$ vanishes. 
Enforcing this boundary condition naturally introduces an additional pressure term to adjust the flow accordingly. 
Introducing pressure and Dirichlet boundary conditions has led us to the modified system \eqref{main-sys-2}-\eqref{init-cond}, which has not been considered before to our knowledge. For $\alpha=\beta=0$ it is an exact subsystem of the primitive equation model for the atmosphere, which has motivated us to leverage techniques from this well studied equation. 

Our investigation has focused on two central aspects:  well-posedness and long-term dynamics of the system. In contrast to the known results for the unmodified system, our global well-posedness result for arbitrary initial data in $H^1$ for \eqref{main-sys-2}-\eqref{init-cond}, Theorem~\ref{main-thm}, ensures unique strong solutions that depend continuously on the initial data. This may be seen an indication that top boundary conditions for $v$ and an additional pressure term are natural. 
Based on this, we have considered the long-term behavior of the system. On the one hand we have pointed out runaway modes that for a large set of parameter values: explicit plane wave solutions exhibit exponential and unbounded growth, somewhat similar to those in \cite{Calz}  and to standing modes from rigid lid conditions \cite{Kelly}. On the other hand, for a complementary set of parameter choices, we have shown that the system possesses a compact global attractor. 

Following \cite{MaRo23,AlGr24} we have focussed on the simplified form of the model from \cite{CoJo22} with constant coefficients $\alpha, \beta$ and isotropic diffusion. We expect that our proofs can be readily generalised to the case of $x_2$-dependent coefficients as in \cite{CoJo23}. 

In conclusion, our results offer novel insights into the atmospheric wave model, particularly into long-term behavior, stability and energy dissipation. 
%, thus possibly contributing to a deeper understanding of atmospheric phenomena. 
We have shown that the parameter space supports qualitatively different dynamics for fixed boundary conditions. It would be interesting to explore this further also for other boundary condition such as those suggested recently in \cite{Kelly}. This also concerns the relation to the different wave dynamics that has been identified without boundary conditions in \cite{CoJo22,CoJo23}.

\medskip
\textbf{Acknowledgements.} 
The authors gratefully thank Edriss Titi for pointing out this topic. They also wish to acknowledge the support of the University of Bremen, where a significant portion of this work was completed, and ChatGPT for assistance with translation. This paper is a contribution to Project M1 of the Collaborative Research Centre TRR 181 ``Energy Transfers in Atmosphere and Ocean," funded by the Deutsche Forschungsgemeinschaft (DFG, German Research Foundation) under project number 274762653.

%\\

%\noindent \textbf{Author contributions} Not applicable.\\

%\noindent \textbf{Funding} This paper is a contribution to the project M1 of the Collaborative Research Centre TRR 181 ``Energy Transfers in Atmosphere and Ocean", which is funded by the Deutsche Forschungsgemeinschaft (DFG, German Research Foundation) under project number 274762653.\\

%\noindent \textbf{Availability of data and materials} No new data were generated or analysis is support of this research.\\

%\noindent \textbf{Declarations}\\
%\noindent \textbf{Conflict of interest} On behalf of all authors, the corresponding author states that there is no conflict of
%interest.\\

%\noindent \textbf{Ethics approval} Not applicable.\\

%\noindent \textbf{Consent to participate} Not applicable.\\

%\noindent \textbf{Code availability} Not applicable.

\end{document}